\documentclass[english]{article}
\usepackage[T1]{fontenc}
\usepackage[latin9]{inputenc}
\usepackage{color}
\usepackage{array}
\usepackage{float}
\usepackage{amsthm}
\usepackage{amsmath}
\usepackage{amssymb}
\usepackage{graphicx}
\usepackage{esint}
\usepackage{bbm}
\usepackage{placeins}
\usepackage{verbatim}
\usepackage{enumerate}

\PassOptionsToPackage{normalem}{ulem}
\usepackage{ulem}

\makeatletter

\floatstyle{ruled}
\newfloat{algorithm}{tbp}{loa}
\providecommand{\algorithmname}{Algorithm}
\floatname{algorithm}{\protect\algorithmname}

  \theoremstyle{plain}
  \newtheorem{lem}{\protect\lemmaname}
  \theoremstyle{remark}
  \newtheorem{rem}{\protect\remarkname}
  \theoremstyle{definition}
  \newtheorem{defn}{\protect\definitionname}
\theoremstyle{plain}
\newtheorem{thm}{\protect\theoremname}
  \theoremstyle{plain}
  
 \theoremstyle{definition}
  
    \theoremstyle{plain}

\usepackage{thmtools}

\declaretheoremstyle[
notefont=\bfseries, notebraces={}{},
bodyfont=\normalfont,
postheadspace=0.5em,
numbered=no,
]{mystyle}
\declaretheorem[style=mystyle]{Corollary}

\usepackage {url}
\usepackage [numbers]{natbib}
\usepackage{dsfont}
\usepackage{authblk}

\makeatother

\usepackage{babel}
  \providecommand{\definitionname}{Definition}
  \providecommand{\examplename}{Example}
  \providecommand{\lemmaname}{Lemma}
  \providecommand{\remarkname}{Remark}
\providecommand{\corollaryname}{Corollary}
\providecommand{\theoremname}{Theorem}

\usepackage{hyperref}

\begin{document}

\title{Integral geometry for Markov chain Monte Carlo:\\
overcoming the curse of search-subspace dimensionality}
\date{}

\author{Oren Mangoubi and Alan Edelman} 
\affil{\emph{Department of Mathematics} \\ \emph{Computer Science and Artificial Intelligence Laboratory (CSAIL)}}
\affil{\emph{Massachusetts Institute of Technology}}
\maketitle

\let\thefootnote\relax\footnotetext{Keywords: integral and stochastic geometry, Metropolis-within-Gibbs algorithms, manifold MCMC, Chern-Gauss-Bonnet theorem, Cauchy-Crofton formula, random matrices, real algebraic manifold volume bounds}



\begin{abstract}
We introduce a method that uses the Cauchy-Crofton formula and a new curvature formula from integral geometry to reweight the sampling probabilities of Metropolis-within-Gibbs algorithms in order to increase their convergence speed.  We consider algorithms that sample from a probability density conditioned on a manifold $\mathcal{M}$.  Our method exploits the symmetries
of the algorithms' isotropic random search-direction subspaces to analytically average out the variance in the intersection volume caused by the orientation of the search-subspace with respect to the manifold $\mathcal{M}$ it intersects.  This variance can grow exponentially with the dimension of the search-subspace, greatly slowing down the algorithm.  Eliminating this variance allows us to use search-subspaces of dimensions many times greater than would otherwise be possible, 
 allowing us to sample very rare events that a lower-dimensional search-subspace would be unlikely to intersect.

To extend this method to events that are rare for reasons other than their support $\mathcal{M}$ having a lower dimension, we formulate and prove a new theorem in integral geometry that makes use of the curvature form of the Chern-Gauss-Bonnet theorem to reweight sampling probabilities. On the side, we also apply our theorem to obtain new theoretical bounds for the volumes of real algebraic manifolds.


Finally, we demonstrate the computational effectiveness and
speedup of our method by numerically applying it to the conditional
stochastic Airy operator sampling problem in random matrix theory.
\end{abstract}

\section{Introduction\label{sec:Introduction}}

  Applications of sampling on probability distributions, defined on Euclidean space or on other manifolds, arise in many fields, such as Statistics \cite{Manifold_MCMC_Diaconis,MCMC_review, MCMC_MLE}, Machine Learning \cite{MCMC_Application_Machine_Learning}, Statistical Mechanics \cite{MCMC_manifold_application_statistical_mechanics}, General Relativity \cite{MCMC_Application_General_Relativity}, Molecular Biology \cite{MCMC_Application_molecular_biology}, Linguistics \cite{MCMC_Application_Linguistics}, and Genetics \cite{MCMC_Application_genetics}.  One application of special interest to us is random matrix theory, where we would like to compute statistics for the eigenvalues of random matrices under certain eigenvalue constraints.  In many cases these probability distributions are difficult to sample from with straightforward methods such as rejection sampling because the events we are conditioning on are very rare, or the probability density concentrates in some small regions of space.  Typically, the complexity of sampling from these distributions grows exponentially with the dimension of the space.  In such situations, we require alternative sampling methods whose complexity promises not to grow exponentially with dimension.  In Markov chain Monte Carlo (MCMC) algorithms, one of the most commonly used such methods, we run a Markov chain over the manifold that converges to the desired probability distribution \cite{BayesianBook}.  Unfortunately, in many situations MCMC algorithms still suffer from inefficiencies that cause the Markov chain to have very long (oftentimes exponentially long) convergence times \cite{Graphical_Models_Book, Handbook, MCMC_problems2}.

To illustrate these inefficiencies and our proposed fix, we imagine we would like to sample uniformly  
from a  manifold $\tilde{\mathcal{M}}\subset\mathbb{R}^{n+1}$  (as illustrated in dark blue in  Figure 1.)  By uniformly, we can imagine 
that $\tilde{\mathcal{M}}$ has finite volume, and the probability of being picked in a region
is equal to the volume of that region.  More generally, we can put a probability measure on
$\tilde{\mathcal{M}}$ and sample from that measure.

We consider algorithms that produce a sequence of points $\{x_1, x_2, \ldots\}$ (yellow dots in Figure 1) with the property
that $x_{i+1}$ will be chosen somehow in an (isotropically generated) random plane $S$ (red plane in Figure 1)  centered at $x_i$.  Further,
the step from $x_i$ to $x_{i+1}$ is independent of all the previous steps (Markov chain property.)  This situation is known as a Gibbs sampling Markov chain with isotropic random search-subspaces.

For our purposes, we find it helpful to pick a sphere (light blue) of radius $r$ that represents the length of the jump
we might wish to take upon stepping from $x_i$ to $x_{i+1}$.  Note that $r$ is usually random.
The sphere will be the natural setting to mathematically exploit the symmetries associated with isotropically
distributed planes.  Intersecting with the sphere, the plane $\tilde{S}$ becomes a
great circle $S$ (red), and the manifold $\tilde{\mathcal{M}}$ becomes a submanifold  (blue) of the sphere.
Assuming we take a step length of $r$, then necessarily $x_{i+1}$ must be on the
intersection (green dots in Figure 1, higher-dimensional submanifolds in more general situations) of the red great circle and the blue submanifold.

For definitiveness, suppose our ambient space is $\mathbb{R}^{n+1}$ where $n = 2$, our blue manifold $\tilde{\mathcal{M}}$ has codimension $k=1$, and our search-subspaces have dimension $k+1$.  Our sphere now has dimension $n$ and the great circle dimension $k=1$.  The intersections (green dots) of the great circle with $\mathcal{M}$ are $0$-dimensional points.

We now turn to the specifics of how $x_{i+1}$ may be chosen from the intersection of the red curve and the blue curve.  Every green point is on the intersection of the blue manifold and the red circle.  It is worth pondering the distinction between shallower angles of intersection, and steeper angles.  If we thicken the circle by a small constant thickness $\epsilon$, we see that a point with a shallow angle has a larger intersection than a steep angle.  Therefore points with shallow angles should be weighted more.  Figure \ref{fig:Angle} illustrates that ($\frac{1}{\textrm{sin}(\theta_i)}$) is the proper weighting for an intersection angle of $\theta_i$.

 We will argue that the distinction between shallower and steeper angles takes on a false sense of importance
and traditional algorithms may become unnecessarily inefficient accordingly. 
A traditional algorithm focuses on the specific red circle that happens to be generated by the algorithm and then gives more weight
to intersection points with shallower angles.
We propose that knowledge of the isotropic distribution of the red circle indicates that all angles may be given the same weight.  Therefore, any algorithmic work that goes into weighting points unequally based on the angle of intersection is wasted work.

 Specifically, as we will see in Section \ref{sub:Naive-weights-vs.}, $\frac{1}{\textrm{sin}(\theta_{i})}$ has infinite variance, due in part to the fact that $\frac{1}{\textrm{sin}(\theta_{i})}$ can become arbitrarily large for small enough $\theta_i$.  The algorithm must therefore search through a large fraction of the (green) intersection points before converging because any one point could contain a signifiant portion of the conditional probability density, provided that its intersection angle is small enough.  This causes the algorithm to sample the intersection points very slowly in situations where the dimension is large and there are typically exponentially many possible intersection points to sample from.
 
 This paper justifies the validity of the angle-independent approach through the mathematics of integral geometry \cite{Santalo, Stochastic_Geometry_Book, Gelfand, Helgason, Paiva}, and the Cauchy-Crofton formula in particular in Section \ref{sec:The-First-Order}.  We should note that sampling all the intersection points with equal probability cannot work for just any choice of random search-subspace $S$.  For instance, if the search-subspaces are chosen to be random longitudes on the
2-sphere, parts of $\mathcal{M}$ that have a nearly east-west orientation would be sampled frequently but parts of $\mathcal{M}$
that have nearly north-south orientation would be almost never sampled, introducing a statistical bias to the samples in favor of
\noindent the east-west oriented samples.
However, if $S$ is chosen to be isotropically random, the random orientation of $S$ does not favor either the north-south nor the east-west
parts of $\mathcal{M}$, suggesting that we can sample the intersection points with equal probability in this situation without introducing a bias.  Effectively, by sampling with equal probability weights and isotropic search-subspaces we will use integral
geometry to compute an analytical average of the weights, an average that
we would otherwise compute numerically, thereby freeing up computational
resources and speeding up the algorithm.
\begin{figure}[h]
\includegraphics[scale=0.43]{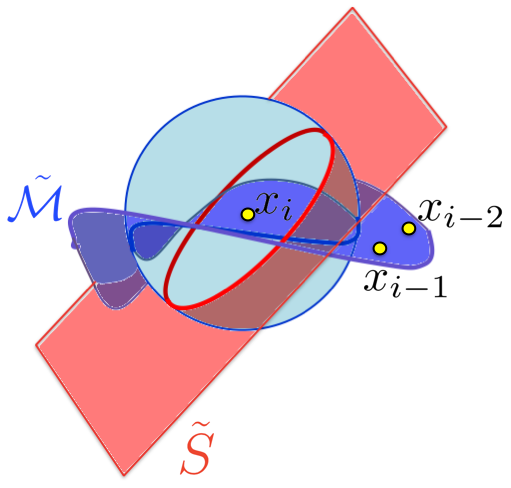}
\includegraphics[scale=0.43]{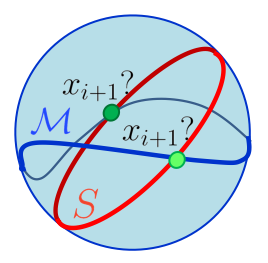}
\caption{In this example we wish to generate random samples on a codimension-$k$ manifold $\tilde{\mathcal{M}}\subset \mathbb{R}^n$ (dark blue) with a Metropolis-within-Gibbs Markov chain $\{x_1,x_2,\ldots\}$ that uses isotropic random search-subspaces $\tilde{S}$ (light red) centered at the most recent point $x_i$ ($k=1$, $n=3$ in figure).  We will consider the sphere $r\mathbb{S}^n$ of an arbitrary radius $r$ centered at $x_i$ (light blue), allowing us to make use of the spherical symmetry in the distribution of the random search-subspace to improve the algorithm's convergence speed.  $\tilde{S}$ now becomes an isotropically distributed random great  $k$-sphere $S=\tilde{S}\cap r\mathbb{S}^n$ (dark red), that intersects a codimension-$k$ submanifold $\mathcal{M}=\tilde{\mathcal{M}}\cap r\mathbb{S}^n$ of the great sphere.\label{fig:Hitandrun}}
\end{figure}

In Part II of this paper, we perform a numerical implementation of an approximate version of the above algorithm in order to sample the eigenvalues of a random matrix conditioned on certain rare events involving other eigenvalues of this matrix.  We obtain different histograms from these samples weighted according to both the traditional weights as well as integral geometry weights  (Figure \ref{fig:intro_plot1}; Figures \ref{fig: 6_eigenvalues} and \ref{fig:large_deviation} in part II).  We find that using integral geometry greatly reduces the variance of the weights.  For instance, the integral geometry weights normalized by the median weight had a sample variance of $3.6\times10^5$, 578, and 1879 times smaller than the traditional weights, respectively, for the top, middle, and bottom simulations of Figure \ref{fig:intro_plot1}.  This reduction in variance allows us to get faster-converging (i.e., smoother for the same number of data points) and more accurate histograms in Figure \ref{fig:intro_plot1}.   In fact, Section \ref{sub:Naive-weights-vs.} shows that the traditional weights have infinite variance due to their second-order heavy tailed probability density, so the sample variance tends to increase greatly as more samples are taken.  Because of the second-order heavy-tailed behavior in the weights, the smoother we desire the histogram to be, the greater the speed up in the convergence time obtained by using the integral geometry weights in place of the traditional weights.

\begin{figure}[h]
\includegraphics[scale=0.19]{Hit_and_run.png}
\includegraphics[trim=1cm 0cm 0cm 0cm, scale=0.45]{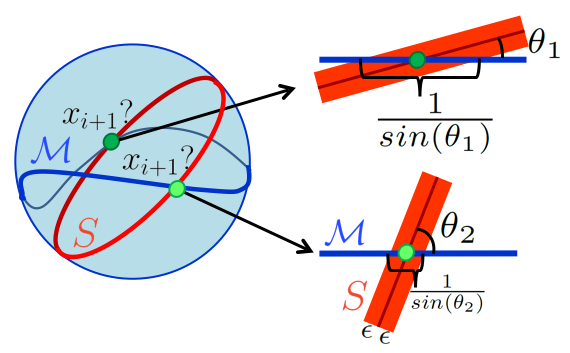}
\caption{Conditional on the next sample point $x_{i+1}$ lying a distance $r$ from $x_i$, the algorithm must randomly choose $x_{i+1}$ from a probability distribution on the intersection points (middle, green) of the manifold $\mathcal{M}$ with the isotropic random great circle $S$ (red).  If traditional Gibbs sampling is used, intersection points with a very small angle of intersection $\theta_{i}$ must be sampled with a much greater (unnormalized) probability $\frac{1}{\textrm{sin}(\theta_{i})}$ (right, top) than intersection points with a large angle (right, bottom).  This greatly increases the variance in the sampling probabilities for different points and slows down the convergence of the method used to generate the next sample $x_{i+1}$.  However, since $S$ is isotropically distributed on $r\mathbb{S}^n$, the symmetry of the isotropic distribution of $S$ allows us to use the Cauchy-Crofton formula from integral geometry to analytically average out these unequal probability weights so that every intersection point now has the same weight, freeing the algorithm from the time-consuming task of effectively computing this average numerically.
\label{fig:Angle}}
\end{figure}

\begin{rem}
 Since we are using an approximate truncated version of the full algorithm that is not completely asymptotically accurate, the integral geometry weights also cause an increase in asymptotic accuracy.  The full MCMC algorithm should have perfect asymptotic accuracy, so we expect this increase in accuracy to become an increase in convergence speed if we allow the Markov chain to mix for a longer amount of time.
 \end{rem}
 
For situations where the intersections are higher-dimensional submanifolds rather than individual points, we show in Section \ref{sec:A-Curvature-Formula} that the angle-independent approach generalizes to a curvature-dependent approach.  We stress that traditional algorithms condition only on the plane that was actually generated while ignoring its isotropic distribution. By taking the isotropy into account, our algorithm can use the curvature information of the manifold to compute an analytical average of the local intersection volumes (local in a second-order sense) with all possible isotropically distributed search-subspaces, greatly reducing the variance of the volumes.

\begin{figure}
\centering
\includegraphics[trim=0cm 2.2cm 0cm 1cm, width=12cm]{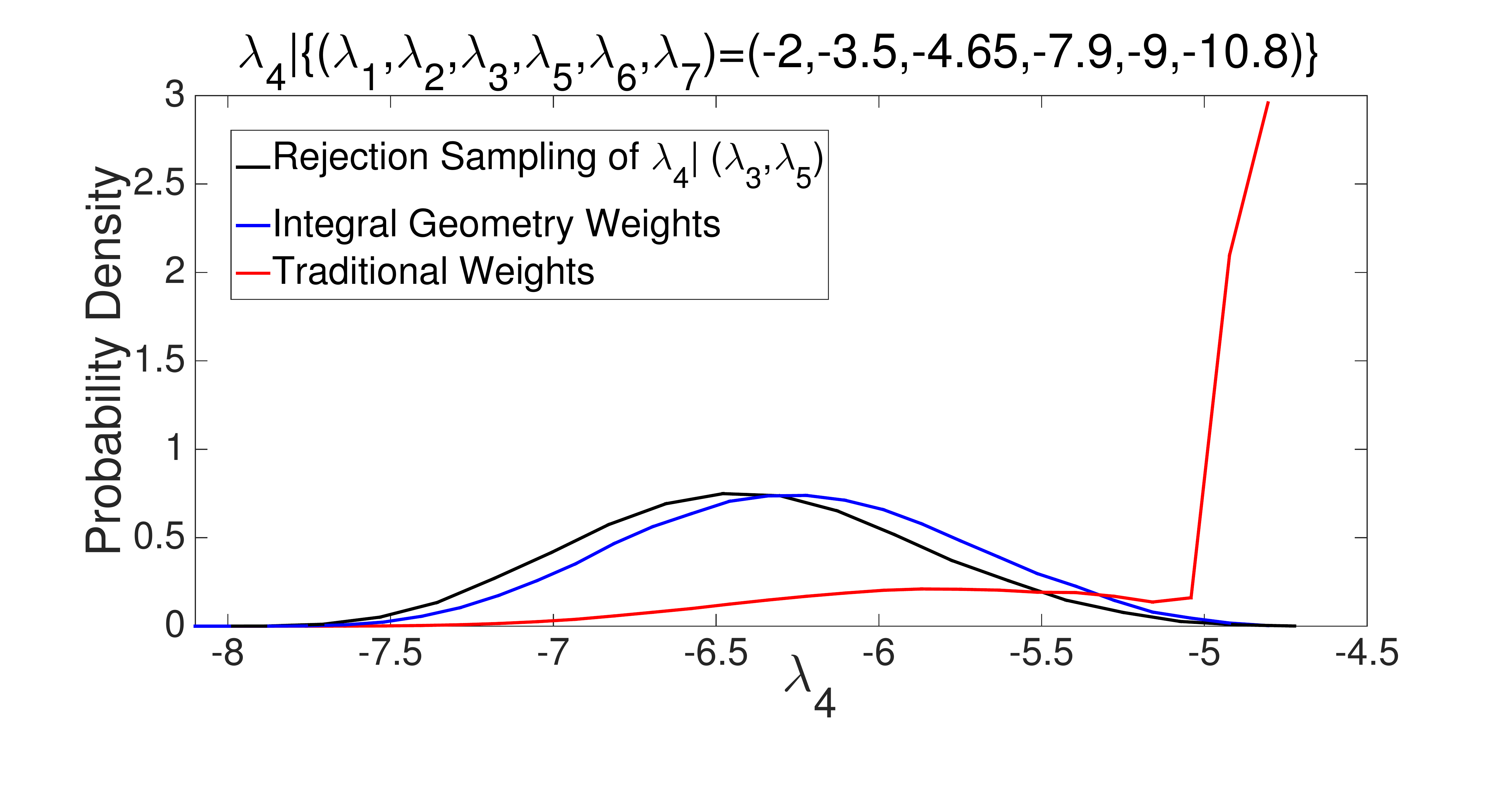}
\includegraphics[trim=0cm 0.3cm 0cm 2cm, clip=true,width=13cm]{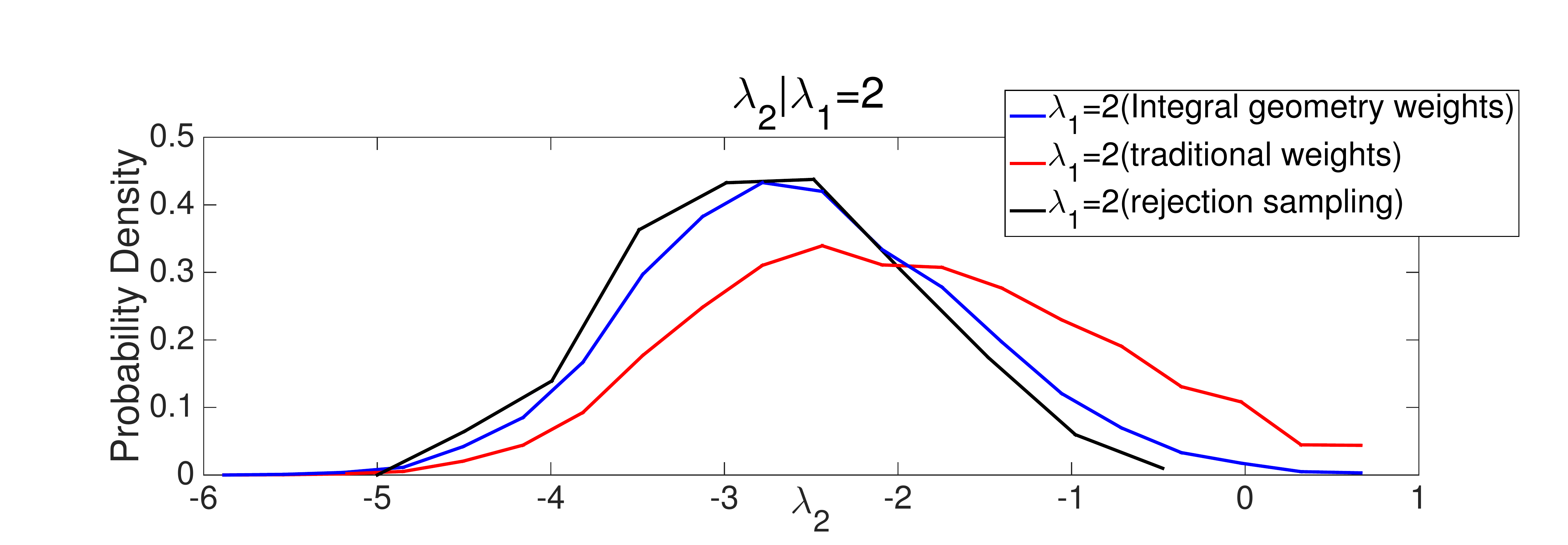}
\includegraphics[trim=0cm 5cm 0cm 5cm, clip=true,width=13cm]{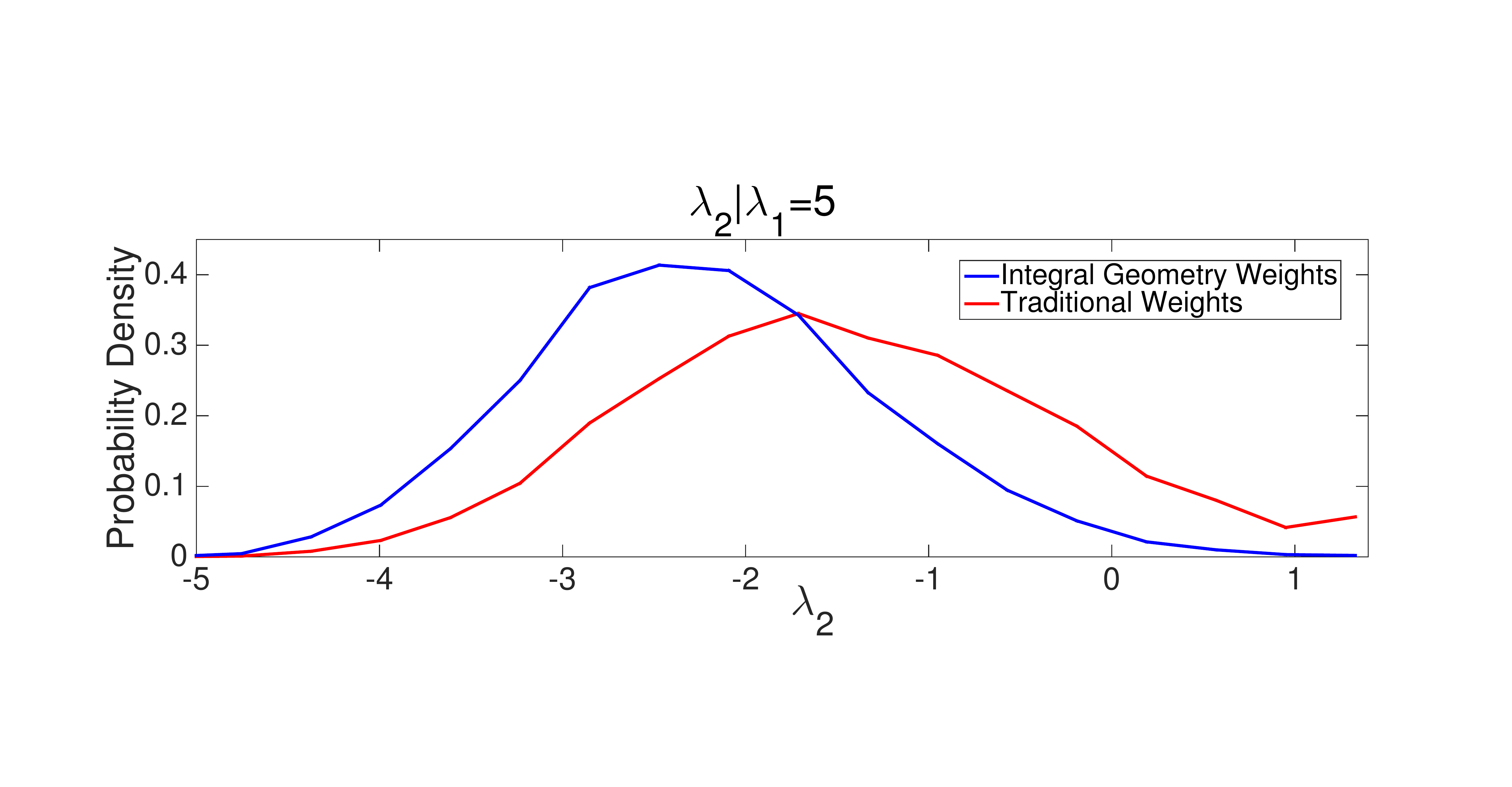}

\caption{Histograms from 3 random matrix simulations (see Sections \ref{sec:Conditioning-on-multiple-eigenvalues} and \ref{sec:Conditioning-on-single-eigenvalue}) where
we seek the distribution of an eigenvalue given conditions on one or more other eigenvalues.
    In all three figures, the blue curve uses the integral geometry weights proposed
in this paper, the red curve uses traditional weights, and the black curve (only in
the top two figures) is obtained by the accurate but very slow rejection sampling method.  
     Two things worth noticing is that the integral geometry weight curve is more accurate than
the traditional weight curve (at least when we have a rejection sampling curve to compare), and that the
integral geometry weight curve is smoother than the traditional weight curve.  
      The integral geometry algorithm achieves these benefits in part because
of the much smaller variance in the weights.  (In these three cases the
integral geometry sample  variance is smaller by a factor of
$\approx$ $10^5$, $600$, and $2000$ respectively)\label{fig:intro_plot1}}
\end{figure}

Higher-dimensional intersections occur in many (perhaps most) situations, such as applications with events that are rare for reasons other than that their associated submanifold has high codimension.  In these situations, the probability of a low-dimensional search-subspace intersecting $\mathcal{M}$ can be very small, so
one may wish to use a search-subspace $S$ of dimension $d$ that is greater
than the codimension $k$ of $\mathcal{M}$ in order to increase the probability
of intersecting $\mathcal{M}$.

As we will see in Section \ref{sub:Sphere-example}, the traditional approach can lead to a huge variance in the intersection volumes that increases exponentially with the difference in dimension $d-k$ (Figure \ref{fig:intro_plot2}, right).   This exponentially large variance leads to the same type of algorithmic slowdowns of the traditional algorithm as the variance in the traditional angle weights discussed above.  Using the curvature-aware approach can oftentimes reduce or eliminate this exponential slowdown.

This paper justifies the validity of the curvature-aware approach by proving a generalization of the Cauchy-Crofton formula (Section \ref{sec:A-Curvature-Formula}).  We then motivate the use of the curvature-aware approach over the traditional curvature-oblivious approach using the mathematics of concentration of measure \cite{Oren_Concentration_of_Measure,Concentration_of_measure_phenomenon, Milman} (Section \ref{sub:Sphere-example}) and differential geometry \cite{SpivakIII,SpivakV}, specifically the Chern-Gauss-Bonnet Theorem \cite{Chern} whose curvature form we use to re-weight the intersection volumes (Section \ref{sub:Optimality-of-Chern-Gauss-Bonnet}).

\begin{figure}
\centering
\includegraphics[scale=0.2]{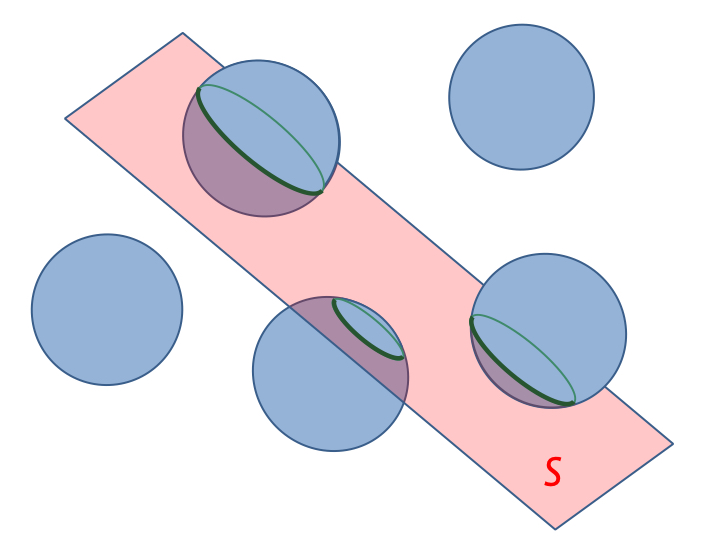}
\includegraphics[trim=3cm 0cm 4cm 0cm, clip=true, scale=0.15]{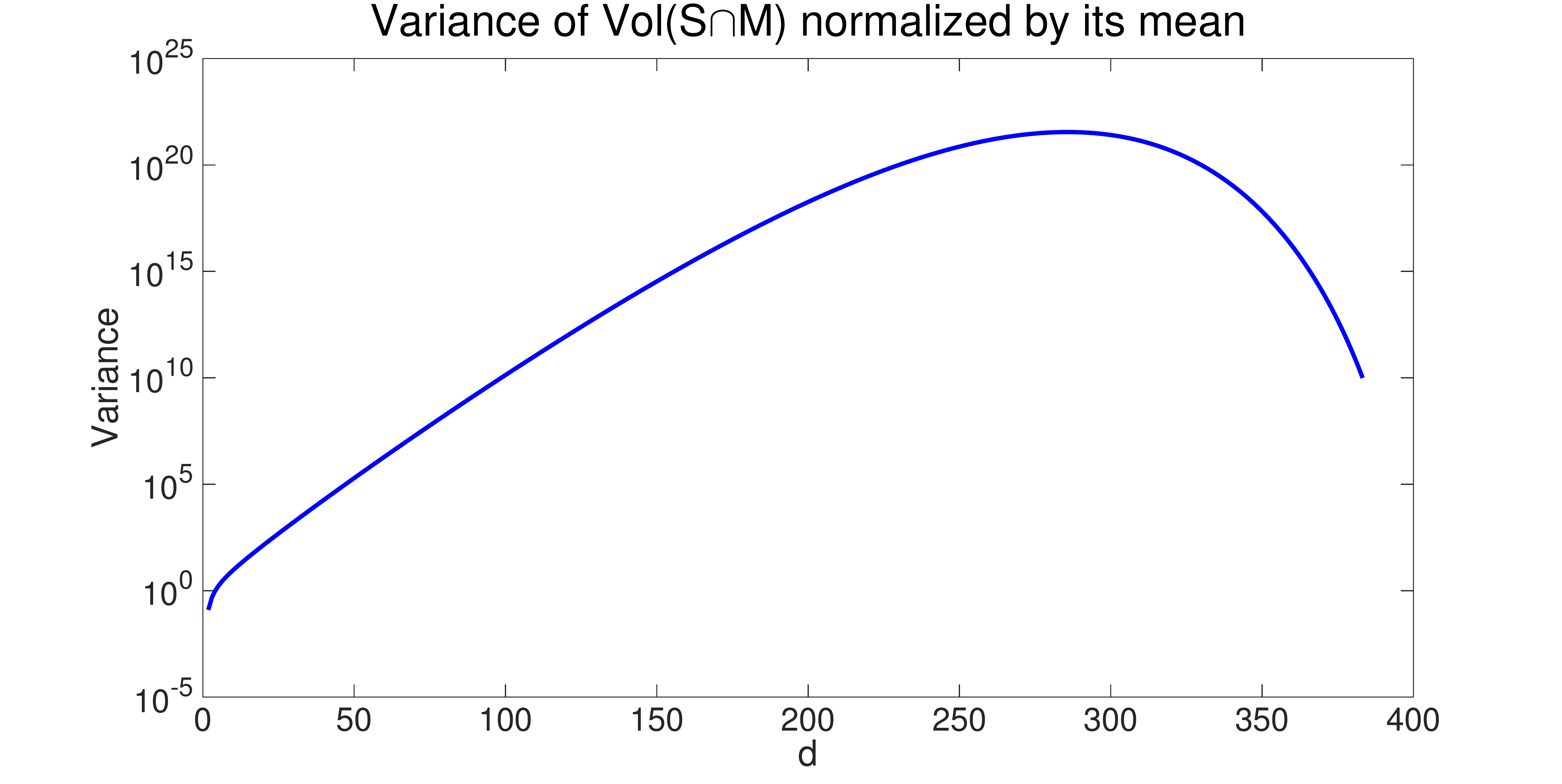}
\caption{In this example a collection $\mathcal{M}$ of $n-1$-dimensional spheres (blue, left) is intersected (intersection depicted as green circles)  by a random search-subspace $S$ (red).  The spheres that $S$ intersects farther form their center will have a much smaller intersection volume than the spheres that $S$ intersects closer to their center, with the variance in the intersection volumes increasing exponentially in the dimension $d$ of $S$ (logarithmic plot, right).   This curse of dimensionality for the intersection volume can lead to an exponential slowdown when using a traditional algorithm to sample from $S \cap \mathcal{M}$.  In Section \ref{sec:A-Curvature-Formula} we will see that this slowdown can be avoided if we use the curvature information to reweight the intersection volumes, reducing the variance in the intersection volumes.  \label{fig:intro_plot2}}
\end{figure}

\part{Theoretical results and discussion}
\section{Integral \& differential geometry preliminaries}

\subsection{Kinematic measure \label{sec:Kinematic_measure}}

Up to this point in the paper we have talked about random search-subspaces informally.  This notion of randomness is formally referred to as the kinematic measure \cite{Santalo, Stochastic_Geometry_Book}.  The kinematic measure provides the right setting to state the Cauchy-Crofton Formula.  The kinematic measure, as the name suggests, is invariant under translations and rotations.

The kinematic measure is the formal way of discussing the following simple situation: we would like to take a random point $p$ uniformly on the unit sphere or, say, inside a cube in $\mathbb{R}^n$.  First we consider the sphere.  After choosing $p$ we then choose an isotropically random plane of dimension $d+1$ through the point $p$ and the center of the sphere.  In the case of the sphere, this is simply an isotropic random plane through the center of the sphere.  On a cube there are some technical issues, but the basic idea of choosing a random point and an isotropic random orientation using that point as the origin persists.  On the cube we would allow any orientation not only those through a "center".  The technical issues relate to the boundary effects of a finite cube or the lack of a concept of a uniform \emph{probability} measure on an infinite space.  In any case the spherical geometry is the natural computational setting because it is compact (If we insist on artificially compactifying $\mathbb{R}^n$ or $\mathbb{H}^n$ by conditioning on a compact subset then either the boundary effects cause the different search-subspaces to vary greatly in volume, slowing the algorithm, or we must restrict ourselves to such a large subset of $\mathbb{R}^n$ or $\mathbb{H}^n$ that most of the search-subspaces don't pass through much of the region of interest).  However, for the sake of completeness we introduce the kinematic measure for all three constant-curvature spaces (spherical, Euclidean, and hyperbolic) because it is relevant in more theoretical applications.

In the spherical geometry case, we define the kinematic measure with respect to a fixed non-random subset $S_\textrm{fixed}\subset \mathbb{S}^n$, usually a great subsphere, by the action of the Haar measure on the special orthogonal group $\mathrm{SO}(n)$ on  $S_\textrm{fixed}$.  When generalizing to Euclidean and hyperbolic geometry, we must be a bit more careful, because there is no uniform probability distribution on $\mathbb{R}^n$ or $\mathbb{H}^n$.  In the case where $S$ has finite $d$-volume, we can circumvent these issues simply by choosing $p$ to be a point in the poisson point process.  To generalize to planes and hyperboloids, we may define the kinematic measure as a poisson-like point process for our search-subspaces with a translationally and rotationally invariant distribution on all of $\mathbb{R}^n$ (or $\mathbb{H}^n$) (the "points" here are the search-subspaces):

\clearpage

\begin{defn}
(Kinematic measure)

Let $\mathbb{K}^n \in \{\mathbb{S}^n, \mathbb{R}^n, \mathbb{H}^n \}$ be a constant-curvature space.  Let $S_\mathrm{fixed}$ be a $d$-dimensional manifold that either has a finite $d$-volume, or is a $plane$ (in $\mathbb{R}^n$ only) or a hyperboloid (in $\mathbb{H}^n$ only).  Let $H$ be the Haar measure on $G$.   If $S$ has finite d-volume we take $G$ to be  the group $I_n$ of isometries of $\mathbb{K}^n$.  If $S$ is a plane or hyperboloid, we instead take G to be the quotient $I_n/I_d$ of the isometries on $\mathbb{K}^n$ with the isometries on $S_\mathrm{fixed}$.    Let $N$ be the counting process such that \\

(i) $\mathbb{E}[N(A)] = \frac{1}{\mathrm{Vol}_d(S_\mathrm{fixed})} \times H(A)$

(ii) $N(A)$ and $N(B)$ are independent\\ \\
for any disjoint Haar-measurable subsets $A,B \subset G$, where we drop the $\frac{1}{\mathrm{Vol}_d(S_\mathrm{fixed})}$ term if $S_\mathrm{fixed}$ is a plane or hyperboloid.  We define the kinematic measure with respect to $S_\textrm{fixed}\subset\mathbb{K}^n$ to be the action of the elements of $N$ on $S_\mathrm{fixed}$.
\end{defn}

If we wish to actually sample from the kinematic measure for the infinite-measure spaces $\mathbb{R}^n$ or $\mathbb{H}^n$ in real life, we must restrict ourselves to some (almost surely) finite subset of the infinite kinematic measure point process. For instance, in this paper we would restrict ourselves to those subspaces that intersect some manifold $\mathcal{M}$ that we would like to sample.

\subsection{The Cauchy-Crofton formula\label{sec:The-Cauchy-Crofton-Formula}}

In this section, we state the Cauchy-Crofton
formula \cite{Crofton, Santalo, Stochastic_Geometry_Book}, which says that the volume of a manifold $\mathcal{M}$ is proportional to the average of the volumes of the intersection $S\cap \mathcal{M}$ of $\mathcal{M}$ with a random kinematic measure-distributed search-subspace $S$.  Our first-order reweigthing (section \ref{sec:The-First-Order}), referred to as the "angle-independent" reweighting in the introduction, is based on this formula.  In Section \ref{sec:A-Curvature-Formula}, we will prove a generalization of this formula that will allow for higher-order reweightings.

\begin{lem}
(Cauchy-Crofton Formula)\label{lem:Crofton}\cite{Crofton, Santalo, Stochastic_Geometry_Book}

Let $\mathcal{M}$ be a codimension-$k$ submanifold of $\mathbb{K}^{n}$, where $\mathbb{K}^{n} \in \{ \mathbb{S}^{n}, \mathbb{R}^{n}, \mathbb{H}^{n} \}$. Let S be a random d-dimensional manifold in $\mathbb{K}^{n}$ of finite volume (or a plane or hyperboloid), distributed according to the Kinematic measure. Then there exists a constant $c_{d,k,n,\mathbb{K}}$ such that

\begin{equation}
\mathrm{Vol}{}_{n-k}(\mathcal{M})=\frac{c_{d,k,n,\mathbb{K}}}{\mathrm{Vol}_{d}(S)}\times\mathbb{E}_{S}[\mathrm{Vol}_{d-k}(S\cap\mathcal{M})],\label{eq:0}
\end{equation}
where we set $\mathrm{Vol}_{d}(S)$ to 1 if  $S$ is a plane or hyperboloid.
In the spherical case we have $c_{d,k,n,\mathbb{S}} = \frac{\mathrm{Vol}(\mathbb{S}^{n-k})\times \mathrm{Vol}(\mathbb{S}^d)}{\mathrm{Vol}(\mathbb{S}^{d-k})}$.
$c_{d,k,n,\mathbb{R}}$ and $c_{d,k,n,\mathbb{H}}$ are given in \cite{Santalo}.

\end{lem}

\subsection{The Chern-Gauss-Bonnet theorem \label{sec:GaussBonnet}}

The Gauss-Bonnet theorem \cite{SpivakIII}, states that the integral of the Gaussian curvature $\mathcal{C}$ of a 2-dimensional manifold $\mathcal{M}$ is proportional to its Euler characteristic $\chi(\mathcal{M})$:

\begin{equation}
\int_\mathcal{M}\mathcal{C}\mathrm{d}A = 2\pi\chi(\mathcal{M}).
\end{equation}

The Chern-Gauss-Bonnet theorem, a generalization of the Gauss-Bonnet theorem to arbitrary even-$m$-dimensional manifolds \cite{Chern,SpivakV}, states that

\begin{equation}
\int_{\mathcal{M}}\textrm{Pf}(\Omega)\mathrm{d}\textrm{Vol}_m = (2\pi)^\frac{m}{2} \chi (\mathcal{M}),
\end{equation}
where $\Omega$ is the curvature form of the Levi-Civita connection and $\textrm{Pf}$ is the pfaffian.  The curvature form $\Omega$ is an intrinsic property of the manifold, i.e., it does not depend on the embedding.  In the special case when $\mathcal{M}$ is a hypersurface, the curvature
$\textrm{Pf}(\Omega_{x})$ may be computed as the Jacobian determinant
of the Gauss map at $x$ \cite{Gauss_Bonnet_for_Hypersurfaces,Gauss_Bonnet_for_Hypersurfaces2},
i.e., as the determinant of the Hessian at $x$ of the manifold when
the orthogonal distance of the manifold to the tangent plane at $x$
is expressed as a function of the tangent space.

The Chern-Gauss-Bonnet theorem is usually viewed as a way of relating the curvature of the manifold with its Euler characteristic.   In Section \ref{sec:A-Curvature-Formula} we will interpret the Chern-Gauss Bonnet theorem as a way of relating the volume form $\mathrm{d}\textrm{Vol}_m$ to the curvature form $\Omega$.  This will come in useful since the curvature form does not change very quickly in sufficiently smooth manifolds, allowing us to get an (in many cases order-of-magnitude) estimate for the volume of the manifold from its curvature form at a single point.

\section{A first-order reweighting via the Cauchy-Crofton formula\label{sec:The-First-Order}}

To simplify the statements of the theorems, we introduce the following
definition:
\begin{defn}
(Unbiased weighting)

We say that the random variable $W$ is an unbiased weighting of a probability measure $\mathbb{P}$ if $\mathbb{P}(A) = \mathbb{E}[W\times \mathbbm{1}_A]$ for every $\mathbb{P}$-measurable set $A$.

\end{defn}
For instance, the weighted mean and the weighted histogram converge
to the same values as the unweighted mean and histogram of $X$ as
the number of samples goes to infinity. The \emph{rate} of convergence,
however, may be very different for the weighted samples than the unweighted
samples.  For example, while the sample means $\frac{1}{n}\sum_{i=1}^n{1+N_i}$ and $\frac{1}{n}\sum_{i=1}^n{1+100\times N_i}$, where $N_1,N_2,... \sim \mathcal{N}(0,1)$ i.i.d., both converge almost surely to $1$ as $n\rightarrow\infty$, $\sum_{i=1}^n{1+100\times N_i}$ converges much slower because the terms have much larger variance.  Our primary goal in this paper is to find weightings that
greatly reduce the variance in the samples and hence greatly increase
the rate of convergence of the estimators.  We now state the main theorem of this section, which uses the Cauchy-Crofton formula to obtain a variance-reducing first-order unbiased weighting of the intersection (Figure \ref{fig:Theorem1}):
\begin{thm}\label{th:1}
Let $\mathbb{Q}$ be the uniform probability measure, with density $f_\mathbb{Q}$, defined on a subset $D\subset\mathbb{K}^{n-1}$ of finite volume.  Let $\lambda:\mathbb{K}^{n}\rightarrow\mathbb{R}^{k}$ be the constraint
function and $a\in\mathbb{R}^{k}$ the constraint value. Let S be
a random search-subspace of dimension $d\geq k$ distributed according to the kinematic measure.
Then the intersection points $x$ of S with the manifold \textup{$\mathcal{M}=\lambda^{-1}(a)$
}can be weighted in an unbiased way with respect to  $f_\mathbb{P}(a)$, the probability density of $\mathbb{P} = \lambda \circ \mathbb{Q}$ at $a$, as

\begin{equation}
w(x)\mathrm{d}\mathrm{Vol}_{d-k}(x)=\frac{c_{d,k,n,\mathbb{K}}}{\mathrm{Vol}_d(S)}\times\frac{f_\mathbb{Q}(x)}{|\nabla(\lambda\circ T)(x)|}\mathrm{d}\mathrm{Vol}{}_{d-k}(x),\label{eq:2}
\end{equation}
where \textup{$\nabla$ }denotes the Jacobian, and\textup{ $|M|:=\sqrt{\textrm{det}(M^{T}M)}$}
denotes the product of the singular values of any matrix $M$.
\end{thm}
\begin{proof}

We first observe that it suffices to prove the theorem for the special case when $\mathbb{Q}$ is the uniform distribution on $D$.  We can then integrate $f_\mathbb{Q}(x) \times \frac{c_{d,k,n,\mathbb{K}}}{\mathrm{Vol}_{n-1}(D)\times\mathrm{Vol}_d(S)}\times\frac{1}{|\nabla \lambda(x)|}$ over $\mathcal{M}$ to extend the result to arbitrary $\mathbb{Q}$.

Let $g\sim\mathbb{Q}$ be a point uniformly
distributed on $D$. Denoting by $B_\epsilon(a)$ the $k$-ball of radius $\epsilon$ centered at $a\in\mathbb{R}^k$, we have
\[
f_\mathbb{P}(a)=\lim_{\epsilon\downarrow0}\frac{\mathbb{\mathbb{P}}(\lambda(g)\in B_{\epsilon}(a))}{\textrm{Vol}_k(B_{\epsilon}(a))}
\]
\[
=\lim_{\epsilon\downarrow0}\frac{\mathbb{\mathbb{P}}(g\in \lambda^{-1}(B_{\epsilon}(a)))}{\textrm{Vol}{}_{k}(B_{\epsilon}(a))}=\lim_{\epsilon\downarrow0}\frac{\textrm{Vol}{}_{n-1}(g\in \lambda^{-1}(B_{\epsilon}(a)))/\textrm{Vol}_{n-1}(D)}{\textrm{Vol}{}_{k}(B_{\epsilon}(a))}
\]
\begin{equation}
=\frac{1}{\textrm{Vol}_{n-1}(D)}\int_{\lambda^{-1}(a)}\frac{1}{|\nabla \lambda(g)|}\mathrm{d}\textrm{Vol}_{n-k-1}(g),
\end{equation}
where the last equality is obtained from the change of variables formula.
We now use the layer cake lemma from measure theory to layer the manifold
$\mathcal{M}=\lambda^{-1}(a)$ with $\frac{1}{|\nabla \lambda(g)|}d\mathrm{Vol}_{n-k-1}$
layers $\mathcal{M}_{y}:=\mathcal{M}\cap\{\frac{1}{|\nabla \lambda(g)|}<y\}$,
and apply the Cauchy-Crofton formula \cite{Crofton} separately to
each of these layers (as illustrated in Figure \ref{fig:Theorem1}):
\begin{equation}
\begin{split}
&\frac{1}{\textrm{Vol}_{n-1}(D)}\int_{\lambda^{-1}(a)}\frac{1}{|\nabla \lambda(g)|}\mathrm{d}\textrm{Vol}{}_{n-k-1}(g)\\
&=\frac{1}{\textrm{Vol}_{n-1}(D)}\int_{y\geq0}\textrm{Vol}{}_{n-k-1}(\mathcal{M}_{y})\mathrm{d}y\\
&=\frac{1}{\textrm{Vol}_{n-1}(D)}\int_{y\geq0}\frac{c_{d,k,n,\mathbb{K}}}{\textrm{Vol}_\mathrm{d}(S)}\times\mathbb{E}_S[\mathrm{Vol}_{d-k}(S\cap\mathcal{M}_{y})]\mathrm{d}y\\
&=\frac{1}{\textrm{Vol}_{n-1}(D)} \times \frac{c_{d,k,n,\mathbb{K}}}{\textrm{Vol}_d(S)}\times\mathbb{E}_S\bigg[\int_{y\geq0}\textrm{Vol}{}_{d-k}(S\cap\mathcal{M}_{y})\mathrm{d}y\bigg]\\
&=\frac{1}{\textrm{Vol}_{n-1}(D)} \times \frac{c_{d,k,n,\mathbb{K}}}{\textrm{Vol}_d(S)}\times\mathbb{E}_S\bigg[\int_{S\cap\mathcal{M}}\frac{1}{|\nabla \lambda(g)|}\mathrm{d}\textrm{Vol}{}_{d-k}\bigg],
\end{split}
\end{equation}
where the expectation $\mathbb{E}_S$ is taken with respect to Kinematic measure on $S$, and $c_{d,k,n,\mathbb{K}}$
is the constant from the Cauchy-Crofton formula.  The exchange of the integral and the expectation holds by the Fubini-Tonelli theorem, since the integrand is nonegative.  Hence, $w(x)=\frac{c_{d,k,n,\mathbb{K}}}{\textrm{Vol}_d(S)}\times\frac{1/\textrm{Vol}_{n-1}(D)}{|\nabla \lambda(x)|}\mathrm{d}\textrm{Vol}{}_{d-k}(x)$
is an unbiased reweighting with respect to $f_\mathbb{P}(a)$.
\end{proof}

\begin{figure}[h]
\centering
\includegraphics[scale=0.4]{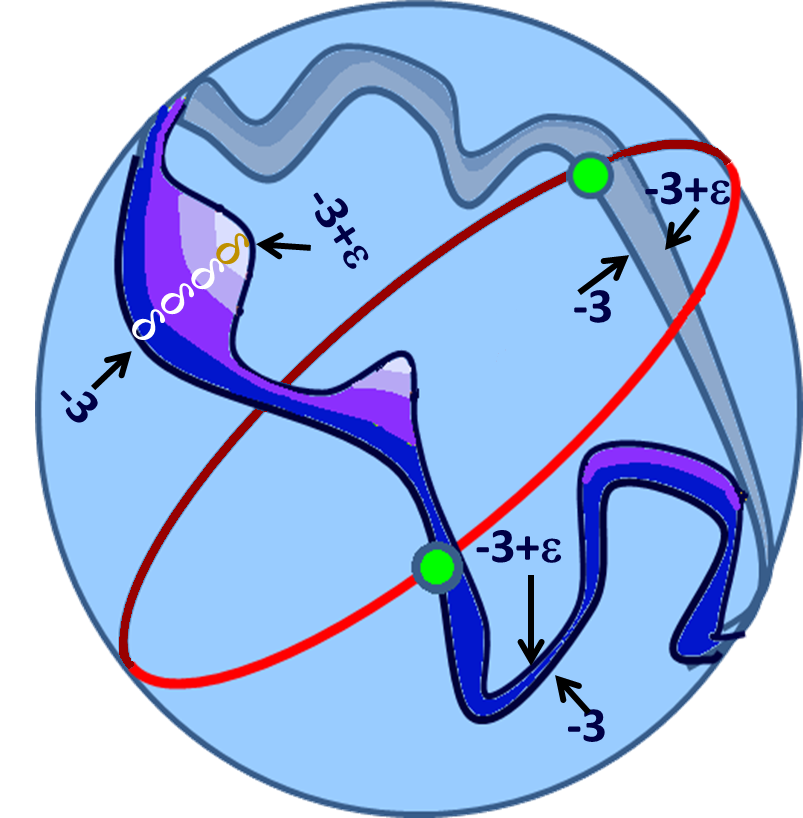}
\caption{The random great circle (red) intersects the constraint manifold (the
blue ribbon which represents the level set $\{g:\lambda_{1}=3\}$
in this example) at different points, generating samples (green dots).
The constraint manifold has different (differential) thickness at
different points, given by $\frac{1}{|\nabla \lambda(g)|}$.
Theorem 1 says that instead of weighting the green dots by the (differential)
intersection length of the great circle and the constraint manifold
at the green dot, we can instead weight it by the local differential
thickness, greatly reducing the variation in the weights (see Sections
\ref{sub:Naive-weights-vs.}, \ref{sec:Conditioning-on-multiple-eigenvalues}
and \ref{sec:Conditioning-on-single-eigenvalue}). Decomposing this
thickness into layers of manifolds, each of uniform thickness, by
means of the Layer Cake Lemma from measure theory, allows us to apply
the Cauchy-Crofton formula individually to each manifold in the proof
of Theorem \ref{th:1} .\label{fig:Theorem1}}
\end{figure}

\subsection{The first-order reweighted algorithm\label{sec:The-reweighted-algorithm}}

As discussed in the introduction, we can apply the first-order reweighting of Theorem 1 to the Metropolis-within-Gibbs algorithm with $d$-dimensional isotropic random search-subspaces to get a more efficient MCMC algorithm (Algorithm \ref{alg:Hit-and-Run}):

\begin{algorithm}[H]
\caption{Integral Geometry reweighted Metropolis-within-Gibbs MCMC\label{alg:Hit-and-Run}}

\begin{enumerate}
\item \textbf{Input:} Oracle for Probability density $f:\mathbb{R}^{n}\rightarrow[0,1]$,
Oracle for Constraint function $\lambda:\mathbb{R}^{n}\rightarrow\mathbb{R}^{k}$,
$n\geq k\geq0$, (we condition on $\mathcal{M}=\{\lambda(x)=c\}$)

\item  \textbf{Input:} An oracle for the Jacobian $\nabla\lambda$
\item \textbf{Input:} Oracle for observed statistic $\psi:\mathbb{R}^{n}\rightarrow\mathbb{R}^{s}$
\item \textbf{Input:} Search-subspace dimension $d\geq k$, Starting point
$x_{0}$, probability density of $\rho(r)$ of for the step distance $r$, number of Gibbs sampling iterations $i_{\textrm{max}}$
\item For $i=1$ to $i_{\textrm{max}}$

\begin{enumerate}
\item Generate a random isotropic $d$-dimensional linear search-subspace
$S_{i+1}$ centered at $x_{i}$ (this can be easily
done using spherical Gaussians and the QR \cite{Trefethen_book} decomposition)
\item Use an MCMC method (usually heavily based on a nonlinear solver, as in \cite{Optimization_in_MCMC}) to
sample a point $x_{i+1}$ from the (unnormalized) probability density
\begin{equation}
w(x)=\frac{f(x)\times \rho(||x-x_i||)}{|\nabla\lambda|_{\mathcal{S}_x}(x)|}\mathrm{d}\textrm{Vol}_{d-k}
\end{equation}
supported on $S_{i+1}\cap\mathcal{M}$, where $\nabla\lambda|_{\mathcal{S}_x}$ is the gradient of the restriction of $\lambda$ to the sphere $\mathcal{S}_{x}$ of radius $||x-x_{i}||$ centered at $x_{i}$. (If $\mathcal{M}$ is full-dimensional then$\frac{1}{|\nabla\lambda|_{\mathcal{S}}(x)|}$ is set to 1)

(Note: This is
the "Metropolis" step in the traditional Metropolis-within-Gibbs
algorithm, but reweighted according to Theorem \ref{th:1} restricted to the sphere $\mathcal{S}_{x}$)
\item compute $\psi(x_{i})$
\end{enumerate}
\item \textbf{Output: }Unweighted samples $\{x_{i}\}_{i=1}^{i_{\textrm{max}}}$
asymptotically distributed as $i\rightarrow\infty$ according to
the conditional density $f|\{\lambda(x)=c\}$, and $\psi(x_{1}),\psi(x_{2}),...,\psi(x_{i_{\textrm{max}}})$ (from which we can compute
statistics of $\psi$, such as the mean, variance, or the histogram
of $\psi$) \end{enumerate}
\end{algorithm}

In many cases, we can take the probability density $f$ to be spherical
Gaussian (for instance, $\lambda$ and $\psi$ can be functions of
a random matrix whose entries are functions of iid Gaussians $x=(x_{1},...,x_{n}))$.
In this situation, we only need to perform one search-subspace iteration
to obtain a sample from the correct distribution (Algorithm \ref{alg:Great-Sphere-Metropolis}):

\begin{algorithm}[H]
\caption{Integral Geometry reweighted independent search-subspace Metropolis-within-Gibbs
MCMC for sampling from functions of Gaussians\label{alg:Great-Sphere-Metropolis}}

Goal: We wish to condition on $\mathcal{M}=\{\lambda(x)=c\}$, where
the probability distribution on $\mathbb{R}^{n}$ is $f(x) = \frac{1}{\sqrt{(2\pi)^{n}}}e^{-\frac{1}{2}x^{T}x}$,
the density of iid standard normals.
\begin{enumerate}
\item \textbf{Input:} Oracle for Constraint function $\lambda:\mathbb{R}^{n}\rightarrow\mathbb{R}^{k}$,
$n\geq k\geq0$
\item  \textbf{Input:} Oracle for the Jacobian $\nabla\lambda$
\item \textbf{Input:} Oracle for observed statistic $\psi:\mathbb{R}^{n}\rightarrow\mathbb{R}^{s}$
\item \textbf{Input:} Search-subspace dimension $d\geq k$.  Number of iterations $i_\mathrm{max}$.
\item for $i=1$ to $i_{\textrm{max}}$

\begin{enumerate}
\item Generate a random isotropic $d$-dimensional linear search-subspace $S_{i}$ centered at the origin.

\item Use an MCMC method (usually heavily based on a nonlinear solver, as in \cite{Optimization_in_MCMC})
to sample a point $x_{i}$ from the (unnormalized) probability density $w(x)=\frac{f(x)\times \rho_{\chi_n}(||x||)}{|\nabla\lambda|_{\mathcal{S}_x}(x)|}$
supported on $S_{i+1}\cap\mathcal{M}$. $\rho_{\chi_n}$ is the density of the $\chi_n$ distribution and $\nabla\lambda|_{\mathcal{S}_x}$ is the gradient of the restriction of $\lambda$ to the sphere $\mathcal{S}_{x}$ of radius $r = ||x||$ centered at the origin.  (If $\mathcal{M}$ is full-dimensional then$\frac{1}{|\nabla\lambda|_{\mathcal{S}_x}(x)|}$
is set to 1.)

\end{enumerate}

\item \textbf{Output: }Unweighted samples $\{x_{i}\}_{i=1}^{i_{\textrm{max}}}$
that are independent and correctly distributed even for finite $i$
according to the conditional density $f|\{\lambda(x)=c\}$, from which we can obtain $\psi(x_{1}),\psi(x_{2}),...,\psi(x_{i_{\textrm{max}}})$
(and compute statistics of $\psi$, such as the mean, variance, or histogram of $\psi$)\end{enumerate}
\end{algorithm}

\subsection{Traditional weights vs. integral geometry weights\label{sub:Naive-weights-vs.}}
In this section we find the theoretical distribution of the traditional weights and compare them to the integral geometry weights of Theorem \ref{th:1}. We will see that while the traditional weights have an infinite variance, greatly slowing the MCMC algorithm, the integral geometry weights vary only with the differential thickness of the level set $\mathcal{M} = \{x: \lambda(x) = a\}$.

In the codimension-$k=1$ case, we can find the distribution of the weights by observing that the symmetry of the Haar measure means that the distribution of the weights are a local property that does not depend on the choice of manifold $\mathcal{M}$.  Moreover, since the Kinematic measure is locally the same for all three constant curvature spaces $\mathbb{S}^n$, $\mathbb{R}^n$, and $\mathbb{H}^n$, the distribution is the same regardless of the choice of constant curvature space.  Hence, without loss of generality, we may choose $\mathcal{M}$ to be the unit circle in $\mathbb{R}^2$.  Because of the rotational symmetry of both the kinematic measure and the circle, without loss of generality we may condition on only the vertical lines $\{(x,t) : t\in\mathbb{R}\}$, in which case $x$ is distributed uniformly on $[-1,1]$.  The weights are then given by $w = w(x) = \sqrt{1+\frac{x^2}{1-x^2}}$, with exactly two intersections at almost every $x$.  Hence, $\mathbb{E}[w]= 2\int_{-1}^1{\sqrt{1+\frac{x^2}{1-x^2}}}\mathrm{d}x = 2\pi$, the circumference of the circle, as expected.  However, $\mathbb{E}[w^2]= 2\int_{-1}^1{1+\frac{x^2}{1-x^2}}\mathrm{d}x =\infty$.  Hence, the weights $w$ have infinite variance, greatly slowing the convergence of the sampling algorithm even in the codimension-$k=1$ case!  On the other hand, the integral geometry weights, being identically $=1$ have variance zero, so the weights do not slow down the convergence at all.  (A related computation, which we do not give here, shows that the theoretical weights for general $k$ are given by the Wishart matrix determinant $|\frac{1}{\textrm{det}(G^{T}G)}|$, where $G$ is a $(k+1)\times k$ matrix of iid standard normals, which also has infinite variance.)

In practice, nonlinear solvers do not find the different intersection points uniformly at random, so different points can have a different distribution of weights, introducing an inaccuracy in the estimator that uses our samples.  As we saw in Figure \ref{fig:intro_plot1}, the inaccuracy (as well as the variance) is much greater when using the traditional weights than when using the integral geometry weights.  This inaccuracy should ideally be corrected by randomizing the solver by turning it into a Markov chain.  The greater the randomization needed, the more the solver behaves like a random walk and less like a solver, slowing the convergence \cite{Handbook}. Since the samples paired with the traditional weights have much greater inaccuracies that need to be corrected, a Markov chain using the traditional weights will require greater randomization of the nonlinear solver (in addition to having a much greater variance in the weights), and hence should converge much more slowly than a Markov chain using the traditional weights.

\FloatBarrier

\section{A second-order reweighting via the Chern-Gauss-Bonnet theorem \label{sec:A-Curvature-Formula}}

Oftentimes, it is necessary to use a random great sphere of dimension
$d$ larger than the codimension $k$ of the constraint manifold.
For instance, the manifold might represent a rare event, so we might
use a higher dimension than the codimension to increase the probability
of finding an intersection with the manifold. However, the intersections
will no longer be points but submanifolds of dimension $d-k$. How
should one assign weights to the points on this submanifold? The first-order
factor in this weight is simple: it is the same as the Jacobian weight
of Equation \ref{eq:2}. However, the size of the intersection still
depends on the orientation of the great sphere with respect
to the constraint manifold. For instance, we will see in Section \ref{sub:Sphere-example} that if we intersect a sphere
with a plane near its center, then we will get a much larger intersection
than if we intersect the sphere with a plane far from its center.

This example suggests that we should weight the points on the intersection
using the local curvature form, which is described by the second derivatives
of the function whose level set is the constraint manifold: If we
intersect in a direction where the second derivative is greater (with
the plane not passing near the center in the example) then we should
use a larger weight than in directions where the second derivative is smaller
(when the plane passes near the center) (Figure \ref{fig:GaussBonnet}).

\begin{figure}
\centering
\includegraphics[trim=0cm 1.5cm 0cm 1.2cm, clip=true, scale=0.23]{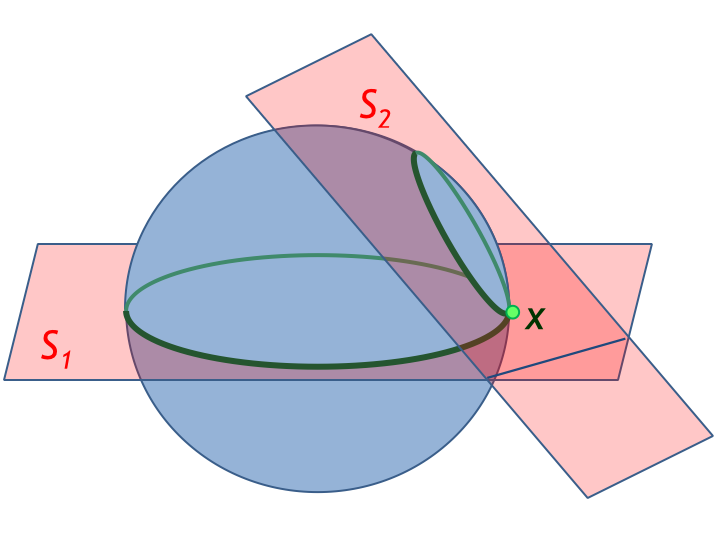}
\caption{Both d-dimensional slices, $S_{1}$ and $S_{2}$, pass through the
green point $x$, but the slice passing through the center of the
n-1 sphere $\mathcal{M}$ has a much bigger intersection volume than
the slice passing far from the center. The smaller slice also has
larger curvature at any given point $x$. If we reweight the density
of $S_{i}\cap\mathcal{M}$ at $x$ by the Chern-Gauss-Bonnet curvature
of $S_{i}\cap\mathcal{M}$ at $x$, then both slices will have exactly
the same total reweighted volume (exact in this case since the sphere
has constant curvature form), since the Chern-Gauss-Bonnet theorem relates
this curvature to the volume measure. \label{fig:GaussBonnet}}
\end{figure}

Consider the simple case where $\mathcal{M}$ is a collection of spheres.  If we were just applying an algorithm based on Theorem \ref{th:1}, such as Algorithm \ref{alg:Hit-and-Run}, we would sample uniformly from the volume on the intersection $S \cap \mathcal{M}$ (Step 6 in Algorithm \ref{alg:Hit-and-Run}). However, the intersected volume depends heavily on the orientation of the search-subspace $S$ with respect to each intersected sphere (Figure \ref{fig:ConcentrationOfMeasure_Euclidean}), meaning that the algorithm will in practice have to search through exponentially many spheres before converging to the uniform distribution on $S \cap \mathcal{M}$  (See section \ref{sub:Sphere-example}).  To avoid this problem, we would like to sample from a density $\hat{w}$ that is proportional to the absolute value of the Chern-Gauss-Bonnet curvature of $S\cap \mathcal{M}$ at each point $x$ in the intersection: $\hat{w} = \hat{w}(x; S) = |\textrm{Pf}(\Omega_x(S\cap\mathcal{M}))|$ (The motivation for using the Chern-Gauss-Bonnet curvature $\textrm{Pf}(\Omega_x(S\cap\mathcal{M}))$ will be discussed in Section \ref{sub:Motivation_for_curvature_reweighting}).

  However, sampling from the density $\hat{w}(x; S)$ does not in general produce unbiased samples uniformly distributed on $\mathcal{M}$ even when $S$ is chosen at random according to the kinematic measure.  We will see in Theorem \ref{th:2} that in order to guarantee an unbiased uniform sampling of $\mathcal{M}$ we can instead sample from the normalized curvature density
  
  \begin{equation}
  w(x; S) =  \frac{\mathrm{Vol}_d(S)}{c_{d,k,n,\mathbb{K}}} \times \frac{\hat{w}(x; S)}{\mathbb{E}_Q\big[\hat{w}(x; S_Q) \times \det(\mathrm{Proj}_{\mathcal{M}_x^\perp} Q)\big]}.
  \end{equation} 
The normalization term $\mathbb{E}_Q\big[\hat{w}(x; S_Q) \times \det(\mathrm{Proj}_{\mathcal{M}_x^\perp} Q)\big]$ is the average curvature at $x$ over all the random orientations at which $S$ \emph{could have} passed through $x$.  Here $S_Q = Q(S-x)+x$ is a random isotropically distributed rotation of $S$ about $x$, with $Q$ the corresponding isotropic random orthogonal matrix.  The determinant inside the expectation is there because while $S$ is originally isotropically distributed, the conditioning of $S$ to intersect $\mathcal{M}$ (at $x$) modifies the probability density of its orientation by a factor of $\det(\mathrm{Proj}_{\mathcal{M}_x^\perp} Q)$.  $P_{\mathcal{M}_x^\perp} Q$ is the projection of the orthogonal complement of the tangent space of $\mathcal{M}$ at x.  In this collection of spheres example, the denominator is a constant for each sphere of a radius $R$.  For instance, in the Euclidean case it can be computed analytically, using the Gauss-Bonnet theorem, as
  \[
  (2\pi)^{\frac{d-1}{2}}2\frac{\Gamma(\frac{d}{2}+1)}{\pi^\frac{d}{2}(n-d)} \times \frac{\Gamma(-\frac{d-1}{2}+1)\Gamma(\frac{n-d}{2}+1)}{(n-d)\Gamma(-\frac{d-1}{2}+\frac{n-d}{2}+1)} R^d.
  \]
From this fact, together with the fact that the total curvature is always the same for any intersection by the Chern-Gauss-Bonnet theorem, we see that when sampling under the probability density $w$ the probability that we will sample from any given sphere is always the same regardless of the volume of the intersection of $S$ with that sphere.  Since each sphere (of the same radius) has an equal probability of being sampled, when sampling from $\mathcal{M}$ the algorithm has to search for far fewer spheres before converging to a uniformly random point on $S \cap \mathcal{M}$ than when sampling from the uniform distribution on $S \cap \mathcal{M}$.


The need to guarantee that $w$ will still allow us to sample uniformly without bias from $\mathcal{M}$ motivates introducing the following theorem (Theorem \ref{th:2}), which, as far as
we know, is new to the literature.  Since the proof does not rely on the fact that $w$ is derived from a curvature form, we state the theorem in a more general form that allows for arbitrary $\hat{w}$ (see Sections \ref{sub:Higher-Order-reweighting} and \ref{sub:Atiyah-Singer_index_theorem} for higher-order choices for $\hat{w}$ beyond just the Chern-Gauss-Bonnet curvature).

\begin{thm} \label{th:2} (Generalized Cauchy-Crofton formula)\\
Let $\mathcal{M}$ be a codimension-$k$ submanifold of $\mathbb{K}^n$ with curvature uniformly bounded above, where $\mathbb{K}^n \in \{\mathbb{S}^n,\mathbb{R}^n,\mathbb{H}^n\}$. Let $S$ be a finite-volume (in $\mathbb{S}^n$, $\mathbb{R}^n$, or $\mathbb{H}^n$), or planar (in $\mathbb{R}^n$), or hyperboloidal  (in $\mathbb{H}^n$) random d-dimensional search-subspace with uniformly bounded curvature distributed according to the kinematic measure. Then the intersection of $S$ with $\mathcal{M}$ can be weighted in an unbiased manner with respect to the volume-measure on $\mathcal{M}$ as
\begin{multline}\label{eq:Th2}
w (x; S)\, \mathrm{d}\mathrm{Vol}{}_{d-k} = \frac{\mathrm{Vol}_d(S)}{c_{d,k,n,\mathbb{K}}} \times \frac{\hat{w}(x; S)}{\mathbb{E}_Q\big[\hat{w}(x; S_Q) \times \det(\mathrm{Proj}_{\mathcal{M}_x^\perp} Q)\big]}\mathrm{d}\mathrm{Vol}{}_{d-k},
\end{multline}
where the pre-normalized weight $\hat{w}(x; S)$ is any function such that $a<\hat{w}(x; S)<b$ for some $0<a<b$, and is Lipschitz in the variable $x \in \mathcal{M}$ for some Lipschitz constant $0<c\neq\infty$ (when using a \underline{translation} of $S$ to keep $x$ in $S \cap \mathcal{M}$ when we vary $x$).

Q is a matrix formed by the first $d$ columns of a random matrix sampled from the Haar measure on $\mathrm{SO}(n)$.  $S_Q = Q(S-x)+x$, where $R_{\perp}$ is a rotation matrix rotating $S-x$ so that it is orthogonal to the tangent space of $\mathcal{M}$. $\mathrm{Proj}_{\mathcal{M}_x^\perp}$ is the projection onto the orthogonal complement of the tangent space of $\mathcal{M}$ at $x$.
 
 (As in Lemma \ref{lem:Crofton}, if $S$ is a plane or hyperboloid, we set $\mathrm{Vol}_d(S)$ to 1.)
\end{thm}


\begin{Corollary}[2.1]\quad

Suppose that $\hat{w}(x; S)$ is $c(t)$-Lipschitz on $\mathcal{M} \cap \{x: \hat{w}(x;S)<t\}$, and that

\begin{equation*}
\lim_{t \rightarrow \infty} \frac{\mathbb{E}_Q\big[\big(\hat{w}(x; S_Q) - \frac{1}{t} \wedge \hat{w}(x; S_Q) \vee t\big) \times \det(\mathrm{Proj}_{\mathcal{M}_x^\perp} Q) \big]}{\mathbb{E}_Q\big[\frac{1}{t} \wedge \hat{w}(x; S_Q) \vee t \times \det(\mathrm{Proj}_{\mathcal{M}_x^\perp} Q) \big]} = 0,
\end{equation*}

and

\begin{equation*}
\lim_{b \rightarrow \infty} \mathbb{E}_S\bigg[\int_{S\cap\mathcal{M}} \mathbbm{1}_A(x) \times [w(x; {S}) - \frac{1}{t} \vee w(x; {S}) \wedge t] \, \mathrm{d}\textrm{Vol}{}\bigg] = 0,
\end{equation*}

where we define the "$\wedge$" and  "$\vee$" operators to be $r\wedge s:= \mathrm{min}\{r, s\}$ and $r\vee s:= \mathrm{max}\{r, s\}$, respectively, for all $r,s \in \mathbb{R}$.

Then Theorem \ref{th:2} holds even for $a = 0$ and $b=c=\infty$.
\end{Corollary}

\vspace{4mm}

\begin{rem}
While the Chern-Gauss-Bonnet curvature pre-weight technically does not satisfy the Lipschitz and boundedness conditions of Theorem \ref{th:2}, we can introduce upper and lower cutoffs $a$ and $b$ to the curvature pre-weight $\hat{w}$ used in the algorithm to make it satisfy these conditions, using the pre-weight $a \vee \hat{w} \wedge b$ instead.  As we shall see in section \ref{sub:Sphere-example}, even in the case of positive-definite curvature, where arbitrarily large intersection curvatures can occur, the volume of the points with curvature larger than a certain cutoff accounts for only a tiny fraction of the average volume of a random intersection.  Hence, introducing an upper cutoff $b$ for the curvature reweighting should only have a tiny effect on the convergence rate, provided that $b$ is large enough (if the curvature form of the manifold is uniformly bounded above, cutting off the curvature pre-weight below $b$ will guarantee that it is Lipschitz as well).  Likewise, if the volume of the points with curvature form below a certain cutoff is very small, then the lower cutoff $a$ should also have a tiny effect on the convergence rate.  For this same reason we expect that for most manifolds of interest the Chern-Gauss-Bonnet curvature pre-weight will satisfy the assumptions of Corollary 2.1, allowing us to use the curvature form without any cutoffs.  Nevertheless, for the sake of completeness, in the future we hope to further weaken the assumptions in Theorem \ref{th:2} beyond what was proved in Corollary 2.1.
\end{rem}

\begin{proof} (Of Theorem \ref{th:2})

We first observe that it suffices to prove Theorem \ref{th:2} for the case where $\mathbb{K}^n =\mathbb{R}^n$ is Euclidean, $S$ is a random plane, and $w(x; S) = w(x; \frac{\mathrm{d}S}{\mathrm{d}\mathcal{M}})$ depends only on the orientation $\frac{\mathrm{d}S}{\mathrm{d}\mathcal{M}} = \frac{\mathrm{d}S}{\mathrm{d}\mathcal{M}}\big|_x$ of the tangent spaces of $S$ and $\mathcal{M}$ at $x$.  This is because constant curvature kinematic measure spaces are locally Euclidean (and converge uniformly to a Euclidean geometry if we restrict ourselves to increasingly small neighborhoods of any point in the space because the curvature is the same).
We may use any geodesic $d$-cube in place of the plane as a search-subspace $S$, since $S$ can be decomposed as a collection of cubes, and Equation \ref{eq:Th2} treats each subset of $S$ in an identical way (since so far we have assumed that $w(x; S)$ depends only on the orientation of the tangent spaces of $S$ and $\mathcal{M}$ at $x$).  We can then approximate any search-subspace $S$ of bounded curvature, and Lipschitz function $w(x; S)$ that depends on the location on $S$ where $S$ intersects $\mathcal{M}$ (in addition to $\frac{\mathrm{d}S}{\mathrm{d}\mathcal{M}}$) by approximating S with very small squares, each with a different "$w(x; \frac{\mathrm{d}S}{\mathrm{d}\mathcal{M}})$" that depends only on  $\frac{\mathrm{d}S}{\mathrm{d}\mathcal{M}}$.

The remainder of the proof consists of two parts.  In Part I we prove the theorem for the special case of very small codimension-$k$ balls (in place of $\mathcal{M}$).  In Part II we extend this result to the entire manifold by tiling the manifold with randomly placed balls.

\vspace{5mm}

\noindent \underline{\textbf{Part I: Special case for small codimension-$k$ balls}}

Let $B_\epsilon=B_\epsilon(x)$ be any k-ball of radius $\epsilon$ that is tangent to $\mathcal{M}\subset \mathbb{R}^n$ at the ball's center $x$.   Let S and $\tilde{S}$ be independent random $d$-planes distributed according to the kinematic measure in $\mathbb{R}^n$.  Let $r$ be the distance in the $k$-plane containing $B_\epsilon$ (the shortest line contained in this plane) from $S$ to the ball's center $x$.  Let $\theta$ be the orthogonal matrix denoting the orientation of $S$.  Then we may write $S = S_{r,\theta}$

Then almost surely (i.e., with probability 1; abbreviated "$a.s.$") $\textrm{Vol}(S_{r,\theta}\cap B_\epsilon)$ does not depend on $\theta$ (this is because $B_\epsilon$ is a codimension-$k$ ball and $S$ is a $d$-plane, so the volume of $S\cap B_\epsilon$, itself a $d-k$-ball, depends $a.s.$ only and $r$ and not on $\theta$).  We also note that $w(x; \frac{\mathrm{d}S_{x,\theta}}{\mathrm{d}B_\epsilon})$ obviously does not depend on $r$ as well.  Define the events $E := \{S_{r,\theta}\cap B_\epsilon \neq \emptyset\}$ and $\tilde{E} := \{\tilde{S}\cap B_\epsilon \neq \emptyset\}$.  Then

\begin{flalign} \label{eq:9}
\mathbb{E}_{r,\theta}\bigg[ w\bigg(x; \frac{\mathrm{d}S_{x,\theta}}{\mathrm{d}B_\epsilon} \bigg) \times \textrm{Vol}_{d-k}(S_{r,\theta}\cap B_\epsilon) \bigg]&&
\end{flalign}

\vspace{-3mm}
\begin{flalign}
&\qquad =\mathbb{E}_{r,\theta}\bigg[ w\bigg(x; \frac{\mathrm{d}S_{x,\theta}}{\mathrm{d}B_\epsilon}\bigg) \times \textrm{Vol}_{d-k}(S_{r,\theta}\cap B_\epsilon) \bigg|E\bigg]\times \mathbb{P}(E)&&
\end{flalign}

\vspace{-3mm}
\begin{flalign}\label{eq:independence}
&\qquad= \mathbb{E}_\theta \bigg[w\bigg(x; \frac{\mathrm{d}S_{x,\theta}}{\mathrm{d}B_\epsilon}\bigg)\bigg|E\bigg]\times \mathbb{E}_r[\textrm{Vol}_{d-k}(S_{r,\theta}\cap B_\epsilon)|E]\times \mathbb{P}(E)&&
\end{flalign}

\vspace{-3mm}
\begin{flalign}\label{eq:13}
&\qquad =  \mathbb{E}_\theta \bigg[\frac{1}{c_{d,k,n,\mathbb{R}}} \times \frac{\hat{w}(x; \frac{\mathrm{d}S_{x,\theta}}{\mathrm{d}B_\epsilon})}{\mathbb{E}_{\tilde{S}}[\hat{w}(x; \frac{\mathrm{d}\tilde{S}}{\mathrm{d}B_\epsilon})|\tilde{E}]}\bigg|E\bigg] \times \mathbb{E}_r[\textrm{Vol}_{d-k}(S_{r,\theta}\cap B_\epsilon)|E] \times \mathbb{P}(E)&&
\end{flalign}

\vspace{-3mm}
\begin{flalign}
&\qquad = \frac{1}{c_{d,k,n,\mathbb{R}}} \times \frac{\mathbb{E}_\theta[\hat{w}(x; \frac{\mathrm{d}S_{x,\theta}}{\mathrm{d}B_\epsilon})|E]}{\mathbb{E}_{\tilde{S}}[\hat{w}(x; \frac{\mathrm{d}\tilde{S}}{\mathrm{d}B_\epsilon})|\tilde{E}]}\times \mathbb{E}_r[\textrm{Vol}_{d-k}(S_{r,\theta}\cap B_\epsilon)|E] \times \mathbb{P}(E)&&
\end{flalign}

\vspace{-3mm}
\begin{flalign}\label{eq:15}
&\qquad=\frac{1}{c_{d,k,n,\mathbb{R}}} \times 1\times  \mathbb{E}_r[\textrm{Vol}_{d-k}(S_{r,\theta}\cap B_\epsilon)|E] \times \mathbb{P}(E)&&
\end{flalign}

\vspace{-3mm}
\begin{flalign}
&\qquad= \frac{1}{c_{d,k,n,\mathbb{R}}} \times \mathbb{E}_{r,\theta}[\textrm{Vol}_{d-k}(S_{r,\theta}\cap B_\epsilon)|E] \times \mathbb{P}(E)&&
\end{flalign}

\vspace{-3mm}
\begin{flalign}
&\qquad=\frac{1}{c_{d,k,n,\mathbb{R}}} \times \mathbb{E}_{r,\theta}[\textrm{Vol}_{d-k}(S_{r,\theta}\cap B_\epsilon)]&&
\end{flalign}

\vspace{-3mm}
\begin{flalign}\label{eq:crofton2}
&\qquad= \frac{1}{c_{d,k,n,\mathbb{R}}} \times c_{d,k,n,\mathbb{R}} \times \textrm{Vol}_{d-k}(B_\epsilon)&&
\end{flalign}

\vspace{-3mm}
\begin{flalign}\label{eq:19}
&\qquad= \textrm{Vol}_{d-k}(B_\epsilon).&&
\end{flalign}


\begin{itemize}

\item Equation \ref{eq:independence} is due to the fact that $r$ and $\theta$ are independent random variables even when conditioning on the event $E$.  This is true because they are independent in the unconditioned kinematic measure on $S$, and remain independent once we condition on $S$ intersecting  $B_\epsilon$ (i.e., the event $E$) because of the symmetry of the codimension-$k$ ball $B_\epsilon$.

\item Equation \ref{eq:13} is due to the fact that, by the change of variables formula,\\
\begin{equation}
\int_{\mathbb{R}^{n-d}} \textrm{Vol}(T_Q + R_{Q^\perp} y \cap B_\epsilon)\mathrm{d}\textrm{Vol}_{n-d}(y)  \times \frac{1}{\det \textrm{Proj}_{B_\epsilon^\perp} Q}= \textrm{Vol}(B_\epsilon)
\end{equation}
for every orthogonal matrix $Q$, where the coordinates of the integral are conveniently chosen with the origin at the center of $B_\epsilon$.  $R_{Q^\perp}$ is rotation matrix rotating the vector $y$ so that it is orthogonal to $T_Q$, the subspace spanned by the rows of $Q$.

Multiplying by $\hat{w}(x; Q)$ and rearranging terms gives
\begin{multline}
\hat{w}(x; Q) \times \det (\textrm{Proj}_{B_\epsilon^\perp} Q)= \\ \hat{w}(x; Q)\times \frac{\int_{\mathbb{R}^{n-d}} \textrm{Vol}(T_Q + R_{Q^\perp} y \cap B_\epsilon)\mathrm{d}\textrm{Vol}_{n-d}(y)}{\textrm{Vol}(B_\epsilon)}.
\end{multline}

Taking the expectation with respect to Q (where Q is the first $d$ columns of a $\mathrm{Haar}(\mathrm{SO}(n))$ random matrix) on both sides of the equation gives
\begin{multline}
\mathbb{E}_Q [\hat{w}(x; Q) \times \det (\textrm{Proj}_{B_\epsilon^\perp} Q)] \\= \mathbb{E}_Q\bigg[\hat{w}(x; Q)\times \frac{\int_{\mathbb{R}^{n-d}} \textrm{Vol}(T_Q + R_{Q^\perp} y \cap B_\epsilon)\mathrm{d}\textrm{Vol}_{n-d}(y)}{\textrm{Vol}(B_\epsilon)}\bigg].
\end{multline}

Recognizing the right hand side as an expectation with respect to the kinematic measure on $T_Q+ R_{Q^\perp} y $ conditioned to intersect $B_\epsilon$ (since the fraction on the RHS is exactly the density of the probability of intersection for a given orientation of Q), we have:

\begin{equation}
\mathbb{E}_Q[\hat{w}(x; Q) \times \det(\textrm{Proj}_{B_\epsilon^\perp} Q)]= \mathbb{E}_{\tilde{S}}\bigg[\hat{w}\bigg(x; \frac{\mathrm{d}\tilde{S}}{\mathrm{d}\mathcal{M}}\bigg)\bigg|\tilde{E}\bigg].
\end{equation}

\item Equation \ref{eq:15} is due to the fact that $\frac{\mathrm{d}S_{x,\theta}}{\mathrm{d}B_\epsilon} = \frac{\mathrm{d}S_{r,\theta}}{\mathrm{d}B_\epsilon}$ because $B_\epsilon$ has a constant tangent space, and hence
\begin{multline}
\mathbb{E}_\theta\bigg[\hat{w}\bigg(x; \frac{\mathrm{d}S_{x,\theta}}{\mathrm{d}B_\epsilon}\bigg)\bigg|E\bigg] = \mathbb{E}_{r, \theta}\bigg[\hat{w}\bigg(x; \frac{\mathrm{d}S_{x,\theta}}{\mathrm{d}B_\epsilon}\bigg)\bigg|E\bigg] \\ = \mathbb{E}_{r,\theta}\bigg[\hat{w}\bigg(x; \frac{\mathrm{d}S_{r,\theta}}{\mathrm{d}B_\epsilon}\bigg)\bigg|E\bigg]  = \mathbb{E}_{\tilde{S}}\bigg[\hat{w}\bigg(x; \frac{\mathrm{d}\tilde{S}}{\mathrm{d}B_\epsilon}\bigg)\bigg|\tilde{E}\bigg].
\end{multline}

\item Equation \ref{eq:crofton2} is by the Cauchy-Crofton formula.
\end{itemize}

Writing $\mathbb{E}_S$ in place of $\mathbb{E}_{r,\theta}$ in Equation \ref{eq:9} (LHS)/ \ref{eq:19} (RHS) (we may do this since $S=S_{r,\theta}$ is determined by $r$ and $\theta$), and observing that $\frac{\mathrm{d}S_{x,\theta}}{\mathrm{d}B_\epsilon} = \frac{\mathrm{d}S_{r,\theta}}{\mathrm{d}B_\epsilon}=\frac{\mathrm{d}S}{\mathrm{d}\mathcal{M}}$, we have shown that
\begin{equation}\label{eq:22}
\mathbb{E}_S\bigg[ w\bigg(x; \frac{\mathrm{d}S}{\mathrm{d}\mathcal{M}}\bigg) \times \textrm{Vol}_{d-k}(S \cap B_\epsilon)\bigg] = \textrm{Vol}_{d-k}(B_\epsilon).
\end{equation}

\vspace{5mm}

\noindent \underline{\textbf{Part II: Extension to all of $\mathcal{M}$}}

All that remains to be done is to extend this result over all of $\mathcal{M}$.  To do so, we consider the Poisson point process $\{x_i^\epsilon\}$ on $\mathcal{M}$, with density equal to $\frac{1}{\textrm{Vol}(B_\epsilon)}$.  We wish to approximate the volume-measure on $\mathcal{M}$ using the collection balls $\{B_\epsilon(x_i^\epsilon)\}$ (think of making a papier-m\^{a}ch\'{e} mold of $\mathcal{M}$ using the balls $B_\epsilon(x_i^\epsilon)$ as tiny bits of paper).

Let $A\subset\mathcal{M}$ be any measurable subset of $\mathcal{M}$.
Since $\mathcal{M}$ and $S$ have uniformly bounded curvature forms, because of the symmetry of the balls and the symmetry of the poisson distribution, the total volume of the balls intersected by $S$ and $A$ converges $a.s.$ to $\textrm{Vol}(S\cap\hat{\mathcal{M}}\cap A)$ on any compact sumbanifold $\hat{\mathcal{M}}\subset \mathcal{M}$:
\begin{equation} \label{eq:23}
\sum_{\{i: x_i^\epsilon \in \hat{\mathcal{M}}\}} \textrm{Vol}(S\cap B_\epsilon(x_i^\epsilon))\times  \frac{\textrm{Vol}(B_\epsilon(x_i^\epsilon)\cap A)}{\textrm{Vol}(B_\epsilon(x_i^\epsilon))} \xrightarrow[\epsilon \downarrow 0]{a.s.}  \textrm{Vol}(S\cap\hat{\mathcal{M}}\cap A)
\end{equation}
and similarly,
\begin{equation} \label{eq:23_2}
\sum_{\{i: x_i^\epsilon \in \hat{\mathcal{M}}\}} \textrm{Vol}(B_\epsilon(x_i^\epsilon)\cap A) \xrightarrow[\epsilon \downarrow 0]{a.s.}  \textrm{Vol}(\hat{\mathcal{M}}\cap A).
\end{equation}

But, by assumption, $w$ is Lipschitz in $x$ on $\mathcal{M}$ (since $\hat{w}$, which appears in both the numerator and denominator of $w$, is Lipschitz, and the denominator is bounded below by $a>0$), so we can cut up $\mathcal{M}$ into a countable union of disjoint compact submanifolds $\sqcup_{j=1}^\infty \mathcal{M}_j$ such that $|w(t; \frac{\mathrm{d}S}{\mathrm{d}\mathcal{M}_j}) - w(x; \frac{\mathrm{d}S}{\mathrm{d}\mathcal{M}_j})| \leq \delta$ on all of $x,t\in \mathcal{M}_j$, and hence, by Equation \ref{eq:23},
\begin{multline}\label{eq:26_1}
\lim_{\epsilon \downarrow 0} \bigg|\sum_{\{i: x_i^\epsilon \in \mathcal{M}_j\}} \textrm{Vol}(S\cap B_\epsilon(x_i^\epsilon)\cap A)\times\frac{\textrm{Vol}(B_\epsilon(x_i^\epsilon)\cap A)}{\textrm{Vol}(B_\epsilon(x_i^\epsilon))} \times w\bigg(x_i^\epsilon; \frac{\mathrm{d}S}{\mathrm{d}\mathcal{M}_j}\bigg) \\ -  \int_{S\cap\mathcal{M}_j \cap A} w\bigg(x; \frac{\mathrm{d}S}{\mathrm{d}\mathcal{M}_j}\bigg)  \mathrm{d}\textrm{Vol}(x)\bigg| \leq \delta \times \textrm{Vol}(S\cap\mathcal{M}_j \cap A)
\end{multline}
$a.s.$ for every $j$.

Summing over all $j$ in equation \ref{eq:26_1} implies that
\begin{multline}\label{eq:27_1}
\lim_{\epsilon \downarrow 0} \bigg|\sum_i \textrm{Vol}(S\cap B_\epsilon(x_i^\epsilon)\cap A)\times \frac{\textrm{Vol}(B_\epsilon(x_i^\epsilon)\cap A)}{\textrm{Vol}(B_\epsilon(x_i^\epsilon))} \times w\bigg(x_i^\epsilon; \frac{\mathrm{d}S}{\mathrm{d}\mathcal{M}}\bigg) \\ -  \int_{S\cap\mathcal{M}\cap A} w\bigg(x; \frac{\mathrm{d}S}{\mathrm{d}\mathcal{M}}\bigg)  \mathrm{d}\textrm{Vol}(x)\bigg| \leq \delta \times \textrm{Vol}(S\cap\mathcal{M}\cap A)
\end{multline}
almost surely.  Since Equation \ref{eq:27_1} is true for every $\delta>0$, we must have that
\begin{multline}\label{eq:26}
\sum_i \textrm{Vol}(S\cap B_\epsilon(x_i^\epsilon)\cap A)\times \frac{\textrm{Vol}(B_\epsilon(x_i^\epsilon)\cap A)}{\textrm{Vol}(B_\epsilon(x_i^\epsilon))} \times w\bigg(x_i^\epsilon; \frac{\mathrm{d}S}{\mathrm{d}\mathcal{M}}\bigg)\\ \xrightarrow[\epsilon \downarrow 0]{a.s}  \int_{S\cap\mathcal{M}\cap A} w\bigg(x; \frac{\mathrm{d}S}{\mathrm{d}\mathcal{M}}\bigg)  \mathrm{d}\textrm{Vol}(x).
\end{multline}

Hence, taking the expectation $\mathbb{E}_S$ on both sides of Equation \ref{eq:26}, we get 
\begin{multline}\label{eq:27}
\mathbb{E}_S \bigg[\sum_i \textrm{Vol}(S\cap B_\epsilon(x_i^\epsilon)\cap A)\times \frac{\textrm{Vol}(B_\epsilon(x_i^\epsilon)\cap A)}{\textrm{Vol}(B_\epsilon(x_i^\epsilon))} \times w\bigg(x_i^\epsilon; \frac{\mathrm{d}S}{\mathrm{d}\mathcal{M}}\bigg)\bigg] \\  \longrightarrow  \mathbb{E}_S \bigg[\int_{S\cap\mathcal{M}\cap A} w\bigg(x; \frac{\mathrm{d}S}{\mathrm{d}\mathcal{M}}\bigg)  \mathrm{d}\textrm{Vol}(x)\bigg]
\end{multline}
a.s. as $\epsilon \downarrow 0$ (we may exchange the limit and the expectation by the dominated convergence theorem, since $|\sum_i \textrm{Vol}(S\cap B_\epsilon(x_i^\epsilon)\cap A)\times w(x_i^\epsilon; \frac{\mathrm{d}S}{\mathrm{d}\mathcal{M}})|$ is dominated by $2\times \textrm{Vol}(S\cap\mathcal{M}) \times \frac{b}{a}$) for sufficiently small $\epsilon$.

Since the sum on the LHS of Equation \ref{eq:27} is of nonnegative terms we may exchange the sum and expectation, by the monotone convergence theorem:
\begin{multline} \label{eq:RHS_limit}
\mathbb{E}_S\bigg[\sum_i \textrm{Vol}(S\cap B_\epsilon(x_i^\epsilon)\cap A)\times \frac{\textrm{Vol}(B_\epsilon(x_i^\epsilon)\cap A)}{\textrm{Vol}(B_\epsilon(x_i^\epsilon))} \times w\bigg(x_i^\epsilon; \frac{\mathrm{d}S}{\mathrm{d}\mathcal{M}}\bigg)\bigg] \\=  \sum_i \mathbb{E}_S\bigg[\textrm{Vol}(S\cap B_\epsilon(x_i^\epsilon)) \times  \frac{\textrm{Vol}(B_\epsilon(x_i^\epsilon)\cap A)}{\textrm{Vol}(B_\epsilon(x_i^\epsilon))}\times w\bigg(x_i^\epsilon; \frac{\mathrm{d}S}{\mathrm{d}\mathcal{M}}\bigg)\bigg].
\end{multline}
But by Equation \ref{eq:22}, $\mathbb{E}_S[\textrm{Vol}(S\cap B_\epsilon(x_i^\epsilon))\times w(x_i^\epsilon; \frac{\mathrm{d}S}{\mathrm{d}\mathcal{M}})] = \textrm{Vol}(B_\epsilon(x_i^\epsilon))$, so
\begin{multline} \label{eq:28}
\mathbb{E}_S\bigg[\sum_i \textrm{Vol}(S\cap B_\epsilon(x_i^\epsilon))\times  \frac{\textrm{Vol}(B_\epsilon(x_i^\epsilon)\cap A)}{\textrm{Vol}(B_\epsilon(x_i^\epsilon))}\times w\bigg(x_i^\epsilon; \frac{\mathrm{d}S}{\mathrm{d}\mathcal{M}}\bigg)\bigg] \\ =  \sum_i \textrm{Vol}(B_\epsilon(x_i^\epsilon))\times  \frac{\textrm{Vol}(B_\epsilon(x_i^\epsilon)\cap A)}{\textrm{Vol}(B_\epsilon(x_i^\epsilon))} \longrightarrow \textrm{Vol}(\mathcal{M}\cap A)
\end{multline}
almost surely as $\epsilon\downarrow 0$ by Equation \ref{eq:23_2}.

Combining Equations \ref{eq:27} and \ref{eq:28} gives

\begin{equation}
\mathbb{E}_S\bigg[\int_{S\cap\mathcal{M}\cap A} w\bigg(x; \frac{\mathrm{d}S}{\mathrm{d}\mathcal{M}}\bigg)  \mathrm{d}\textrm{Vol}(x)\bigg] = \textrm{Vol}(\mathcal{M}\cap A).
\end{equation}

\end{proof}

\begin{proof} (Of Corollary 2.1)

Define
\begin{equation*}
\psi(t):= \frac{\mathbb{E}_Q\big[\big(\hat{w}(x; S_Q) - a \wedge \hat{w}(x; S_Q) \vee b\big) \times \det(\mathrm{Proj}_{\mathcal{M}_x^\perp} Q) \big]}{\mathbb{E}_Q\big[a \wedge \hat{w}(x; S_Q) \vee b \times \det(\mathrm{Proj}_{\mathcal{M}_x^\perp} Q) \big]}.
\end{equation*}

  Let $A$ be any Lebesgue-measurable subset.  Then
  
\begin{flalign}
&\mathbb{E}_S\bigg[\int_{S\cap\mathcal{M}} \mathbbm{1}_A(x) \times w(x; {S}) \, \mathrm{d}\textrm{Vol}{}\bigg]&&
\end{flalign}

\vspace{-3mm}
\begin{flalign}
& =\lim_{t\rightarrow \infty} \mathbb{E}_S\bigg[\int_{S\cap\mathcal{M}} \mathbbm{1}_A(x) \times w(x; {S}) \, \mathrm{d}\textrm{Vol}{}\bigg]&&
\end{flalign}

\vspace{-3mm}
\begin{flalign*}
&= \lim_{t\rightarrow \infty} \mathbb{E}_S\bigg[\int_{S\cap\mathcal{M}} \mathbbm{1}_A(x) \times \frac{1}{t} \vee w(x; {S}) \wedge t\, \mathrm{d}\textrm{Vol}{}\bigg]&&
\end{flalign*}

\vspace{-3mm}
\begin{flalign}
&\qquad \qquad \quad + \lim_{t\rightarrow \infty} \mathbb{E}_S\bigg[\int_{S\cap\mathcal{M}} \mathbbm{1}_A(x) \times [w(x; {S}) - \frac{1}{t}\vee w(x; {S}) \wedge t] \, \mathrm{d}\textrm{Vol}{}\bigg]&&
\end{flalign}

\vspace{-3mm}
\begin{flalign}\label{eq:a}
&= \lim_{t\rightarrow \infty} \mathbb{E}_S\bigg[\int_{S\cap\mathcal{M}} \mathbbm{1}_A(x) \times \frac{1}{t}  \vee w(x; {S}) \wedge t \, \mathrm{d}\textrm{Vol}{}\bigg] + 0&&
\end{flalign}

\vspace{-3mm}
\begin{flalign}
&= \lim_{t\rightarrow \infty} \mathbb{E}_S\bigg[\int_{S\cap\mathcal{M}} \mathbbm{1}_A(x) \times \frac{\frac{1}{t}\vee \hat{w}(x; S) \wedge t}{\mathbb{E}_Q\big[\hat{w}(x; S_Q) \times \det(\mathrm{Proj}_{\mathcal{M}_x^\perp} Q)\big]}\, \mathrm{d}\textrm{Vol}{}\bigg]&&
\end{flalign}

\vspace{-3mm}
\begin{flalign*}
&= \lim_{t\rightarrow \infty} \mathbb{E}_S\bigg[\int_{S\cap\mathcal{M}} \mathbbm{1}_A(x)&&
\end{flalign*}

\vspace{-3mm}
\begin{flalign}
&\qquad \qquad \times \frac{\frac{1}{t}  \vee \hat{w}(x; S) \wedge t}{\mathbb{E}_Q\big[\frac{1}{t}  \vee \hat{w}(x; S_Q) \wedge t \times \det(\mathrm{Proj}_{\mathcal{M}_x^\perp} Q)\big]\times (1 + \psi(t))}\, \mathrm{d}\textrm{Vol}{}\bigg]&&
\end{flalign}

\vspace{-3mm}
\begin{flalign*}
&= \lim_{t\rightarrow \infty} \mathbb{E}_S\bigg[\int_{S\cap\mathcal{M}} \mathbbm{1}_A(x)&&
\end{flalign*}

\vspace{-3mm}
\begin{flalign}
&\qquad \qquad \times \frac{\frac{1}{t}  \wedge \hat{w}(x; S) \vee t}{\mathbb{E}_Q\big[\frac{1}{t}\wedge \hat{w}(x; S_Q) \vee t \times \det(\mathrm{Proj}_{\mathcal{M}_x^\perp} Q)\big]}\, \mathrm{d}\textrm{Vol}{}\bigg]\times \frac{1}{1 + \psi(t)}&&
\end{flalign}

\vspace{-3mm}
\begin{flalign}\label{eq:b}
&= \lim_{t\rightarrow \infty} \mathrm{Vol}(\mathcal{M} \cap A) \times \frac{1}{1 + \psi(t)}&&
\end{flalign}

\vspace{-3mm}
\begin{flalign}
&= \mathrm{Vol}(\mathcal{M}\cap A)\times 1&&
\end{flalign}

\vspace{-3mm}
\begin{flalign}
&= \mathrm{Vol}(\mathcal{M}\cap A),&&
\end{flalign}

\begin{itemize}
\item Equation \ref{eq:a} is true because
 \begin{equation*}
\begin{split}
 0 \leq  \mathbb{E}_S\bigg[\int_{S\cap\mathcal{M}} \mathbbm{1}_A(x) \times \frac{1}{t}\wedge w(x; {S})\vee t \, \mathrm{d}\textrm{Vol}{}\bigg]\\
\leq  \mathbb{E}_S\bigg[\int_{S\cap\mathcal{M}} \frac{1}{t}\wedge w(x; S) \vee t \, \mathrm{d}\textrm{Vol}{}\bigg] \xrightarrow[t\rightarrow\infty]{} 0.
  \end{split}
\end{equation*}

\item Equation \ref{eq:b} follows from Theorem \ref{th:2} using $\frac{1}{t}\wedge \hat{w}(x; S) \vee t$ as our pre-weight.  Indeed, $\frac{1}{t}\wedge \hat{w}(x; S) \vee t$ obviously satisfies the boundedness conditions of Theorem \ref{th:2}.  Moreover, since  $\hat{w}(x;S)$ is $c(t)$-Lipschitz everywhere on $\mathcal{M}$ where $\hat{w}(x;S)<t$, the pre-weight $\frac{1}{t}\wedge \hat{w}(x; S) \vee t$ must be $c(t)$-Lipschitz on all $x \in \mathcal{M}$.
\end{itemize}
\end{proof}

\subsection{Second-order Chern-Gauss-Bonnet theorem reweighted algorithm}

Using the second order Chern-Gauss-Bonnet theorem reweighting of Theorem \ref{th:2} together with the first-order reweighting of Theorem \ref{th:1} (which we already implemented in Algorithm \ref{alg:Hit-and-Run}) gives the following improvement to Algorithm \ref{alg:Hit-and-Run}:

\begin{algorithm}[H]
\caption{Curvature-reweighted Metropolis-within-Gibbs MCMC\label{alg:Curvature}}
All steps except steps 2 and 5(b) are the same as in Algorithm \ref{fig:Hitandrun}.
\begin{enumerate}[2.]
\item \textbf{Input:} An oracle for the Jacobian $\nabla\lambda$ and Levi-Civita connection curvature form $\Omega$ of the level set $\mathcal{M} = \{x: \lambda(x) = c\}$ (possibly given as the set of second partial derivatives)
\end{enumerate}
\begin{enumerate}[5. (b)]
\item Use an MCMC method (usually heavily based on a nonlinear solver, as in \cite{Optimization_in_MCMC}) to
sample a point $x_{i+1}$ from the (unnormalized) probability density
\begin{multline}
w(x)=\frac{f(x)\times  \rho(||x-x_i||)}{|\nabla\lambda|_{\mathcal{S}_x}(x)|} \times\\ \frac{|\textrm{Pf}(\Omega_{x}(S_{i+1}\cap\mathcal{M\cap S}_{x}))|}{\mathbb{E}_Q[|\textrm{Pf}(\Omega_{x}(S_Q\cap\mathcal{M\cap S}_{x}))|\times \det(\mathrm{Proj}_{\mathcal{M}^\perp_x} Q)]}\mathrm{d}\textrm{Vol}_{d-k}
\end{multline}
supported on $S_{i+1}\cap\mathcal{M}$, where $\mathcal{S}_{x}$ is
the sphere of radius $||x-x_{i}||$ centered at $x_{i}$.  (If $\mathcal{M}$ is full-dimensional
then $\frac{1}{|\nabla\lambda|_{\mathcal{S}}(x)|}$ is set to 1)
(Note: This is the "Metropolis" step in the traditional Metropolis-within-Gibbs
algorithm, but reweighted according to Theorems \ref{th:1} and \ref{th:2} restricted
to the sphere $\mathcal{S}_{x}$)
\end{enumerate}
\end{algorithm}

\begin{rem}
The curvature form $\Omega_{x}(S_{i+1}\cap\mathcal{M\cap S}_{x})$ of the intersected manifold
can be computed in terms of the curvature form $\Omega_x(\mathcal{M})$ of the original manifold
by applying the implicit function theorem twice in a row. Also, if
$\mathcal{M}$ is a hypersurface then $|\textrm{Pf}(\Omega_{x}(S_{i+1}\cap\mathcal{M\cap S}_{x}))|$
is the determinant of the product of a random Haar-measure orthogonal
matrix with known deterministic matrices, and hence $\mathbb{E}_Q[|\textrm{Pf}(\Omega_{x}(Q\cap\mathcal{M\cap S}_{x}))|\times \det(\mathrm{Proj}_{\mathcal{M}_x^\perp} Q)]$
is also the expectation of a determinant of a random matrix of this
type. If the Hessian is positive-definite, then we can obtain an analytical
solution in terms of zonal polynomials. Even in the case when the
curvature form is not a positive-definite matrix (it is a matrix with entries
in the algebra of differential forms), the fact that the curvature
form is the Pfaffian of a random curvature form (in particular, a determinant of a real-valued random matrix in the codimension-1 case) should make it very
easy to compute numerically, perhaps by a Monte Carlo method.

This fact also means that it should be easy to bound the expectation,
which allows us to use Theorem \ref{th:2} to get bounds for the volumes of
algebraic manifolds (Section \ref{sub:Algebraic-Geometry}).
\end{rem}

\begin{rem}
While the Chern-Gauss-Bonnet theorem only holds for even-dimensional manifolds, we can always modify the dimension of the search subspace by $1$ so that the dimension $d-k$ of $S\cap\mathcal{M}$ is even.  Since we are sampling from a rare event, we must in any case choose $d>>1$, so it makes little difference computationally if $S$ has dimension $d$ or $d+1$.  Alternatively, we can include a dummy variable to increase both the dimensions $n$ and $d$ by $1$.
\end{rem}

\subsection{Reweighting when sampling from full-dimensional distribution (as
opposed to lower-dimensional manifolds)\label{sub:Full-dimensional-distribution-reweighting}}

In many cases one might wish to sample from a full-dimensional set
of nonzero probability measure. One could still reweight in this situation
to achieve faster convergence by decomposing the probability density
into its level sets, and applying the weights of Theorems \ref{th:1} and \ref{th:2}
separately to each of the (infinitely many) level sets.  We expect this reweighting to speed convergence
in cases where the probability density is concentrated in certain
regions, since when $d$ is large, intersecting these regions with a random search-subspace $S$ typically causes large variations in the integral of the probability density over the different regions intersected by $S$, unless we reweight using Theorems \ref{th:1} and \ref{th:2}.

\subsection{An MCMC volume estimator based on the Chern-Gauss-Bonnet theorem} \label{sub:Volume_MCMC}

In this section we briefly introduce a (as far as we know) new MCMC method of estimating the volume of a manifold that is based on the Chern-Gauss-Bonnet curvature.  While this method is interesting in its own right, we choose to introduce it at this point since it will serve as a good introduction to our motivation (Section \ref{sub:Motivation_for_curvature_reweighting}) for using the Chern-Gauss-Bonnet curvature as a pre-weight for Theorem \ref{th:2}.


Suppose we somehow knew or had an estimate for the Euler characteristic $\chi(\mathcal{M})\neq0$ of a closed manifold $\mathcal{M}$ of even-dimension $m$. We could then use a Markov chain Monte Carlo algorithm to estimate the average Gauss curvature form $\mathbb{E}_{\mathcal{M}}[(\textrm{Pf}(\Omega))]$ on $\mathcal{M}$.

The Chern-Gauss-Bonnet theorem says that
\begin{equation}
\int_{\mathcal{M}}\textrm{Pf}(\Omega)\mathrm{d}\textrm{Vol}_m = (2\pi)^\frac{m}{2} \chi (\mathcal{M}).
\end{equation}
We may rewerite this as
\begin{equation}
\frac{\int_{\mathcal{M}}\textrm{Pf}(\Omega)\mathrm{d}\textrm{Vol}_m}{\int_{\mathcal{M}}\mathrm{d}\textrm{Vol}_m} = \frac{(2\pi)^\frac{m}{2} \chi (\mathcal{M})}{\int_{\mathcal{M}}\mathrm{d}\textrm{Vol}_m}.
\end{equation}
By definition, the left hand side is $\mathbb{E}_{\mathcal{M}}[(\textrm{Pf}(\Omega))]$, and $\int_{\mathcal{M}}\mathrm{d}\textrm{Vol}_m = \textrm{Vol}_m(\mathcal{M})$, so
\begin{equation}
\mathbb{E}_{\mathcal{M}}[(\textrm{Pf}(\Omega))] = \frac{(2\pi)^\frac{m}{2} \chi (\mathcal{M})}{\textrm{Vol}_m(\mathcal{M})},
\end{equation}
from which we may derive an equation for the volume in terms of the known quantities $\mathbb{E}_{\mathcal{M}}[(\textrm{Pf}(\Omega))]$ and $\chi (\mathcal{M})$
\begin{equation}\label{eq:Volume_MCMC}
\textrm{Vol}_m(\mathcal{M}) = \frac{(2\pi)^\frac{m}{2} \chi (\mathcal{M})}{  \mathbb{E}_{\mathcal{M}}[(\textrm{Pf}(\Omega))]}.
\end{equation}

\subsection{Motivation for reweighting with respect to Chern-Gauss-Bonnet curvature\label{sub:Optimality-of-Chern-Gauss-Bonnet}}\label{sub:Motivation_for_curvature_reweighting}

While Theorem \ref{th:2} tells us that any pre-weight $\hat{w}$ generates an unbiased weight $w$, it does not tell us what pre-weights reduce the variance of the intersection volumes.  We argue here that the Chern-Gauss-Bonnet theorem in many cases provides us with an ideal pre-weight if one only has access to the local second-order information at a point $x$.

Equation \ref{eq:Volume_MCMC} of Section \ref{sub:Volume_MCMC} gives an estimate for the volume 

\begin{equation}\label{eq:volume_preweight}
\textrm{Vol}_{d-k}(S\cap\mathcal{M}) = \frac{(2\pi)^\frac{d-k}{2} \chi (S\cap\mathcal{M})}{  \mathbb{E}_{S\cap\mathcal{M}}[(\textrm{Pf}(\Omega(S\cap\mathcal{M})))]},
\end{equation}
where $\Omega(S\cap\mathcal{M})$ is the curvature form of the submanifold $S\cap\mathcal{M}$.

If we had access to all the quantities in Equation \ref{eq:volume_preweight} our pre-weight would then be $\frac{1}{\textrm{Vol}_{d-k}(S\cap\mathcal{M})} = \frac{  \mathbb{E}_{S\cap\mathcal{M}}[(\textrm{Pf}(\Omega(S\cap\mathcal{M})))]}{(2\pi)^\frac{d-k}{2} \chi (S\cap\mathcal{M})}$.  However, as we shall see we cannot actually implement this pre-weight since some of these quantities represent higher-order information.  To make use of this weight to the best of our ability given only the second-order information, we must separate the higher-order components of the weight from the second-order components by dividing out the higher-order components.

The Euler characteristic is essentially a higher-order property, so it is not reasonable in general to try to estimate the Euler characteristic $\chi (S\cap\mathcal{M})$ using the second derivatives of $\mathcal{M}$ at $x$ because the local second order information gives us little if any information about $\chi (S\cap\mathcal{M})$ (although it may in theory be possible to say a bit more about the Euler characteristic if one has some prior knowledge of the manifold).  The best we can do at this point is to assume the Euler characteristic is a constant with respect to S, or more generally, statistically independent of $S$.

All that remains to be done is to estimate $\mathbb{E}_{S\cap\mathcal{M}}\textrm{Pf}(\Omega(S\cap\mathcal{M}))$.  We observe that
\begin{equation}
\mathbb{E}_{S\cap\mathcal{M}}\textrm{Pf}(\Omega(S\cap\mathcal{M}))= \mathbb{E}_{S\cap\mathcal{M}}|\textrm{Pf}(\Omega(S\cap\mathcal{M}))|  \times \frac{\mathbb{E}_{S\cap\mathcal{M}}\textrm{Pf}(\Omega(S\cap\mathcal{M}))}{\mathbb{E}_{S\cap\mathcal{M}}|\textrm{Pf}(\Omega(S\cap\mathcal{M}))|}.
\end{equation}
But the ratio $\frac{\mathbb{E}_{S\cap\mathcal{M}}\textrm{Pf}(\Omega(S\cap\mathcal{M}))}{\mathbb{E}_{S\cap\mathcal{M}}|\textrm{Pf}(\Omega(S\cap\mathcal{M}))|}$ is also a higher-order property since all it does is describe how much the second-order Chern-Gauss-Bonnet curvature form changes globally over the manifold, so in general we can say nothing about it using only the local second-order information.  The best we can do at this point is to assume that this ratio is statistically independent of $S$ as well.

Hence, we have:
\begin{multline}
\frac{1}{\textrm{Vol}_{d-k}(S\cap\mathcal{M})}=\mathbb{E}_{S\cap\mathcal{M}}|\textrm{Pf}(\Omega(S\cap\mathcal{M}))|  \times\\ ((2\pi)^\frac{d-k}{2} \chi (S\cap\mathcal{M}) \frac{\mathbb{E}_{S\cap\mathcal{M}}\textrm{Pf}(\Omega(S\cap\mathcal{M}))}{\mathbb{E}_{S\cap\mathcal{M}}|\textrm{Pf}(\Omega(S\cap\mathcal{M}))|}),
\end{multline}
where we lose nothing by dividing out the unknown quantity\\ $(2\pi)^m \chi (\mathcal{M}) \frac{\mathbb{E}_{S\cap\mathcal{M}}\textrm{Pf}(\Omega(S\cap\mathcal{M}))}{\mathbb{E}_{S\cap\mathcal{M}}|\textrm{Pf}(\Omega(S\cap\mathcal{M}))|}$ since we have no information about it and it is independent of $S$.
 
 We would therefore like to use $\mathbb{E}_{S\cap\mathcal{M}}|\textrm{Pf}(\Omega(S\cap\mathcal{M}))|$ as a pre-weight.  Since we only know the curvature form $\Omega(S\cap\mathcal{M})$ locally at $x$, our best estimate for $\mathbb{E}_{S\cap\mathcal{M}}|\textrm{Pf}(\Omega(S\cap\mathcal{M}))|$ is the absolute value $|\textrm{Pf}(\Omega_x(S\cap\mathcal{M})|$ of the Chern-Gauss-Bonnet curvature at $x$.  Hence, our best local second-order choice for the pre-weight is $\hat{w} = |\textrm{Pf}(\Omega_x(S\cap\mathcal{M})|$.

\subsection{Higher-order Chern-Gauss-Bonnet reweightings \label{sub:Higher-Order-reweighting}}
One may consider higher-order reweightings which attempt to guess not only the second-order local intersection volume, but also make a better guess for both the Euler characteristic of the intersection $S_Q \cap \mathcal{M}$, and how the curvature would vary over $S_Q \cap \mathcal{M}$.  Nevertheless, higher-order approximations are probably harder to implement for the same reason that most nonlinear solvers, such as Newton's method, do not use higher-order derivatives.

\subsection{Possible reweightings using Atiyah-Singer index theorem or other topological invariants}\label{sub:Atiyah-Singer_index_theorem}
One may also consider reweighting with respect to topological invariants of Riemannian manifolds other than the Chern-Gauss-Bonnet curvature.  For instance, it may be possible to reweight with respect to the integrand of the Atiyah-Singer index theorem \cite{Singer}, which is the product of the Chern-Gauss-Bonnet curvature form and another differential form associated with an elliptical partial differential equation (PDE) defined on $\mathcal{M}$.  The Atiyah-Singer index theorem says that the integral of the product of these two differential forms over $\mathcal{M}$ is equal to the product of $\chi(\mathcal{M})$ and another term that is invariant under continuos transformations of the PDE.  The idea would be to use a carefully chosen PDE, whose associated differential form attempts to "counterbalance" the curvature form: when the manifold's curvature is big, the PDE's differential form is small, and vice versa.  However, it remains to be shown whether such elliptical PDEs are easy to obtain for an implicitly defined manifold $\mathcal{M}$.

\subsection{Collection-of-spheres example and concentration of measure\label{sub:Sphere-example}}

In this section we argue that the traditional algorithm suffers from an exponential slowdown (exponential in the search-subspace dimension) unless we reweight the intersection volumes using Corollary 2.1 with the Chern-Gauss-Bonnet curvature weights.  We do so by applying two concentration of measure results, which we derive in \cite{Oren_Concentration_of_Measure}, to an example involving a collection of hyperspheres.

Consider a collection of very many hyperspheres in $\mathbb{R}^n$.  We wish to sample uniformly from these hyperspheres.  To do so, we imagine running a Markov chain with isotropically random search-subspaces.  We imagine that there are so many hyperspheres that a random search-subspace typically intersects exponentially many hyperspheres.  As a first step we would use Theorem \ref{th:1} which allows us to sample the intersected hypersphere from the uniform distribution on their intersection volumes.  While using Theorem \ref{th:1} should speed convergence somewhat (as discussed in Section \ref{sub:Naive-weights-vs.}), concentration of measure causes the intersections with the different hyperspheres to have very different volumes (Figure \ref{fig:ConcentrationOfMeasure_Euclidean}).  In fact we shall see that the variance of these volumes increases exponentially in $d$, causing an exponential slowdown if only Theorem \ref{th:1} is used, since the Metropolis subroutine would need to find exponentially many subspheres before converging.

Reweighting the intersection volumes using Theorem \ref{th:2} causes each random intersection $S\cap\mathcal{M}_i$ (where $\mathcal{M}_i$ is a subsphere) to have exactly the same reweighted intersection volume, regardless of the location where $S$ intersects $\mathcal{M}_i$, and regardless of $d$.  Hence, in this example, Theorem \ref{th:2} allows us to avoid the exponential slowdown in the convergence speed that would otherwise arise from the variance in the intersection volumes.

\begin{figure}[h]
\centering
\includegraphics[trim=0cm 0.3cm 0cm 0.6cm, clip=true, scale=0.25]{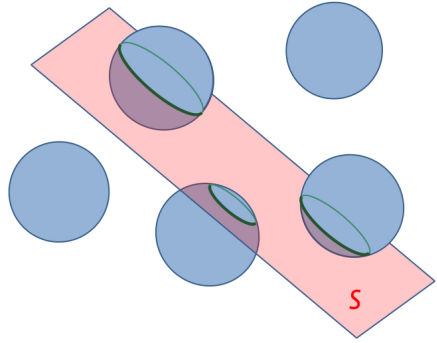}
\caption{The random search-subspace $S$ intersects a collection of spheres $\mathcal{M}$.  Even though the spheres in this example all have the same $n-1$-volume, the $d-1$-volume of the intersection of $S$ with each individual sphere (green circles) varies greatly depending on where $S$ intersects the sphere if $d$ is large.  In fact, the variance of the intersection volume of each intersected sphere increases exponentially with $d$.  This "curse of dimensionality" for the intersection volume variance leads to an exponential slowdown if we wish to sample from $S\cap \mathcal{M}$ with a Markov chain sampler (and $S\cap \mathcal{M}$  consists of exponentially many intersected spheres).  However, if we use the Chern-Gauss-Bonnet curvature to reweight the intersection volumes, then all spheres in this example will have exactly the same reweighed intersection volume, greatly increasing the convergence of the Markov chain sampler. \label{fig:ConcentrationOfMeasure_Euclidean}}
\end{figure}

The first result deals with the variance of the intersection volumes of a sphere in Euclidean space.  It says that the variance of the intersection volume, normalized by it's mean, increases exponentially with the dimension $d$ (as long as $d$ is not too close to $n$).  Although isotropically random search-subspaces are (conditional on the radial direction) distributed according to the Haar measure in spherical space, the Euclidean case is still of interest to us since it represents the limiting case when the hyperspheres are small, since spherical space is locally Euclidean.

\begin{thm}\label{th:3} (Concentration of Euclidean Kinematic Measure)\\
Let $S\subset\mathbb{R}^{n}$ be a random $d$-dimensional linear affine subspace distributed according to the Kinematic measure on $\mathbb{R}^{n}$.  Let $\mathcal{M}=\mathbb{S}^n\subset\mathbb{R}^{n}$ be the unit sphere in $\mathbb{R}^{n}$.  Defining $\alpha:=\frac{d}{n}$, we have
\begin{equation}\label{eq:euclidean_variance_concentration}
k(\alpha,d)e^{d \times \varphi(\alpha)} - 1 \leq \mathrm{Var}(\frac{\mathrm{Vol}(S\cap\mathcal{M})}{\mathbb{E}[\mathrm{Vol}(S\cap\mathcal{M})]}) \leq K(\alpha,d)e^{d \times \varphi(\alpha)} - 1,
\end{equation}
where
\begin{equation*}
\begin{split}
&\varphi(\alpha) = \log(2) + (\frac{1}{\alpha})\log(\frac{1}{\alpha}) - (\frac{1}{2\alpha}+\frac{1}{2})\log(\frac{1}{\alpha}+1) - (\frac{1}{2\alpha}-\frac{1}{2})\log(\frac{1}{\alpha}-1)\\
&k(\alpha,d)=(\frac{(2\pi)^{\frac{3}{2}}}{e^4})(n-d)^2\sqrt{\frac{(n-1)(\frac{n}{d}-1)}{(d-1)(n+d-2)}} \times e^{- 1   -\frac{1+\alpha}{1+\alpha-\frac{2}{d}}}\\
&K(\alpha,d) = (\frac{e^3}{4\pi^2})(n-d)^2 \sqrt{\frac{(n-1)(\frac{n}{d}-1)}{(d-1)(n+d-2)}} \times e^{-\frac{n}{n-1} + 1}.
\end{split}
\end{equation*}
\end{thm}

The next result (Figure \ref{fig:ConcentrationOfMeasure_Spherical}) deals with the spherical geometry case.    As in the Euclidean case, the spherical concentration result says that the variance of the intersection volume increases exponentially with the dimension $d$ as well.  (While we were able to derive the analytical expression for the probability distribution of the intersection volumes, which we used to generate the plot in Figure \ref{fig:ConcentrationOfMeasure_Spherical} showing an exponential increase in variance, we have not yet finished deriving a formal inequality analogous to Theorem \ref{th:3} for the spherical geometry case.  We hope to make the analogous result available soon in \cite{Oren_Concentration_of_Measure})

\begin{figure}[h]
\centering
\includegraphics[scale=0.2]{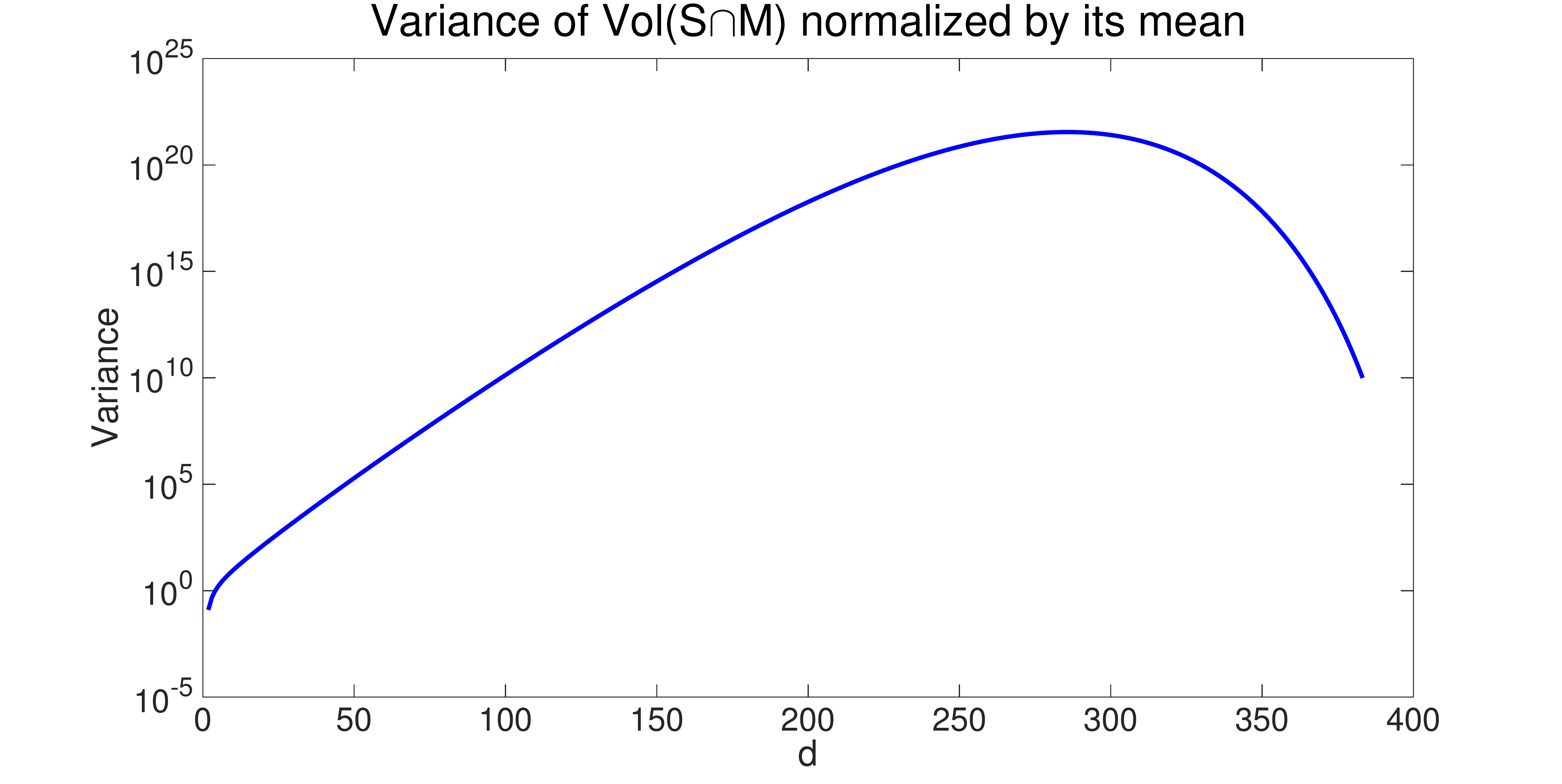}

\caption{This log-scale plot (from \cite{Oren_Concentration_of_Measure})  shows the variance of $\mathrm{Vol}(S\cap\mathcal{M})$ normalized by its mean, when $S$ is a Haar-measure distributed d-dimensional great subsphere of $\mathbb{S}^n$, for different values of $d$, where $n=400$. $\mathcal{M}$ is taken to be the boundary of a spherical cap of $\mathbb{S}^n$ with geodesic radius such that $S$ has a  $10\%$ probability of intersecting $\mathcal{M}$.  The variance increases exponentially with the dimension $d$ of the search-subspace (as long as $d$ is not too close to $n$), leading to an exponential slow-down in the convergence for the traditional Metropolis-within-Gibbs algorithm applied to the collection-of-spheres example of Section \ref{sub:Sphere-example}.  Reweighting the intersection volumes with the Chern-Gauss-Bonnet curvature using Corollary 2.1 in this situation (where $\mathcal{M}$ is a subsphere) causes each random intersection $S\cap\mathcal{M}$ to have exactly the same intersection volume regardless of $d$, allowing us to avoid the exponential slowdown in the convergence speed that would otherwise arise from the variance in the intersection volumes.\label{fig:ConcentrationOfMeasure_Spherical}}
\end{figure}

\subsection{Theoretical bounds derived using Theorem \ref{th:2} and algebraic geometry\label{sub:Algebraic-Geometry}}

Generalizing on bounds for lower-dimensional algebraic manifolds based on the Cauchy-Crofton formula (such as the bounds for tubular neighborhoods in \cite{Tubular_Neighborhoods} and \cite{Guth}), it is also possible to use Theorem \ref{th:2} to get a bound for the volume of an algebraic manifold $\mathcal{M}$ of given degree $s$, as long as one can also use analytical arguments to bound the second-order Chern-Gauss-Bonnet curvature reweighting factor on $\mathcal{M}$ for some convenient search-subspace dimension $d$:

\begin{Corollary}[2.2]
Let $\mathcal{M}\subset\mathbb{R}^{n}$
be an algebraic manifold of degree $s$ and codimension 1, such that
$\mathbb{E}_Q[|\textrm{Pf}(\Omega_x(S_Q \cap\mathcal{M})|\times \det(\mathrm{Proj}_{\mathcal{M}_x^\perp} Q)]\geq b$
for every $x\in\mathcal{M}$, and the conditions of Corollary 2.1 are satisfied if we set $\hat{w}(x;S) = |\textrm{Pf}(\Omega_x(S\cap\mathcal{M})|$. Then 

\textup{
\begin{equation}
\textrm{Vol}(\mathcal{M})\leq \frac{1}{c_{d,k,n,\mathbb{R}}} \times \frac{1}{b} \times\frac{s\times(s-1)^{d}}{2}\textrm{Vol}(\mathbb{S}^{n}).
\end{equation}
}
\end{Corollary}

\begin{proof}
If we have an algebraic manifold of degree $s$ in $\mathbb{R}^{n}$,
by Bezout's theorem the intersection with an arbitrary plane is also
degree $s$. Hence (at least in the case where $\mathcal{M}$ has
codimension 1), we can use Risler's bound to bound the integral of
the absolute value of the Gaussian curvature over $S\cap\mathcal{M}$
by $a:=\frac{s\times(s-1)^{d}}{2}\textrm{Vol}(\mathbb{S}^{n})$ \cite{Total_Absolute_Curvature_1,Total_Absolute_Curvature_2}.

By Theorem \ref{th:2},

\begin{equation*}
\begin{split}
&\textrm{Vol}(\mathcal{M}) = \mathbb{E}_S[ \frac{\textrm{Vol}_d(S)}{c_{d,k,n,\mathbb{K}}}
\times \frac{|\textrm{Pf}(\Omega_x(S\cap\mathcal{M})|}{\mathbb{E}_Q[|\textrm{Pf}(\Omega_x(S_Q\cap\mathcal{M})|\times \det(\mathrm{Proj}_{\mathcal{M}_x^\perp} Q)]}\mathrm{d}\textrm{Vol}{}_{d-k}] \\
&\leq \frac{1}{c_{d,k,n,\mathbb{R}}} \times \mathbb{E}_S[ |\textrm{Pf}(\Omega_x(S\cap\mathcal{M})|]\times \frac{1}{b} \leq \frac{1}{c_{d,k,n,\mathbb{R}}} \times a \times \frac{1}{b}.
\end{split}
\end{equation*}
\end{proof}

Unlike a bound derived using only the Cauchy-Crofton formula for point
intersections, the bound in Corollary 2.2 allows us to incorporate additional information
about the curvature, so we suspect that this bound will be much stronger in situations where the curvature does not vary too much in most directions over the manifold.  We hope to investigate examples of such manifolds in the future where we suspect Corollary 2.2 will provide stronger bounds, but do not pursue these examples here because it is beyond the scope of this paper.

\part{Numerical simulations}

\section{Random matrix application: Sampling the stochastic airy operator\label{sec:Random-Matrix-Application}}

Oftentimes, one would like to know the distribution of the largest
eigenvalues of a random matrix in the large-n limit, for instance
when performing principal component analysis \cite{Johnstone}. For a large class
of random matrices that includes the Gaussian orthogonal/unitary/symplectic
ensembles, and more generally the beta-ensemble point processes, the joint distribution of the largest eigenvalues converges in the large-$n$ limit, after rescaling,
to the so-called hard-edge limiting distribution (The single-largest eigenvalue's
limiting distribution is the well-known Tracy-Widom distribution) \cite{Johnstone,Sutton_Thesis,Stochastic_Operators,RRV}.
One way to learn about these distributions is to generate samples
from certain large matrix models. One such matrix model that converges
particularly fast to the large-n limit is the tridiagonal matrix discretization
of the Stochastic Airy operator of Edelman and Sutton \cite{Sutton_Thesis,Stochastic_Operators},
\begin{equation}
\frac{\mathrm{d}^{2}}{\mathrm{d}x^{2}}-x+\frac{2}{\sqrt{\beta}}\mathrm{d}W\label{eq:Stochastic-Operator},
\end{equation}
where $dW$ is the white noise process.  We wish to study the distributions of eigenvalues of the hard edge
conditioned on other eigenvalue(s) or eigenvector statistics.

To obtain samples from these conditional distributions, we can use Algorithm \ref{alg:Great-Sphere-Metropolis}, which is straightforward to apply in
this case since $dW$ is already discretized
as iid Gaussians.

The stochastic Airy operator (\ref{eq:Stochastic-Operator}) can be
discretized as the tridiagonal matrix \cite{Stochastic_Operators,Sutton_Thesis}

\begin{equation}
A_{\beta}=\frac{1}{h^{2}}\Delta-h\times\textrm{diag}(1,2,...,k)+\frac{2}{h\sqrt{\beta}}N,\label{eq:8}
\end{equation}
where $\Delta=\left[\begin{array}{rrrrrr}
-2 & 1\\
1 & -2 & 1\\
 & 1 & -2 & 1\\
 &  &  & ...\\
 &  &  & 1 & -2 & 1\\
 &  &  &  & 1 & -2
\end{array}\right]$ is the $k\times k$ discretized Laplacian, $N\sim\textrm{diag}\mathcal{\mathrm{(}N}(0,1)^{n})$
is a vector of independent standard normals, and the cutoff $k$
is chosen (as in \cite{Stochastic_Operators,Sutton_Thesis}) to be
$k=10n^{-\frac{1}{3}}$ (the $O(10n^{-\frac{1}{3}})$ cutoff is due
to the decay of the eigenvectors corresponding to the largest eigenvalues,
which decay like the Airy function, causing only the first $O(10n^{-\frac{1}{3}})$ entries to be computationally significant).

\subsection{Approximate sampling algorithm implementation\label{sub:Approximate-sampling-algorithm-implementation}}

Since the discretized stochastic Airy operator $A_{\beta}$ is already
explicitly a function of the iid Gaussians $\frac{2}{h\sqrt{\beta}}N$
(Equation \ref{eq:8}), we can use Algorithm \ref{alg:Great-Sphere-Metropolis}
to sample $A_{\beta}$ conditional on our eigenvalue constraints of
interest. To simplify the algorithm we can use a deterministic
nonlinear solver with random starting points in place of the nonlinear solver-based MCMC ``Metropolis'' subroutine of Algorithm \ref{alg:Great-Sphere-Metropolis}  to get
an approximate sampling (Algorithm 2.1, 
below). This is somewhat analogous
to setting both the hot and cold baths in a simulated annealing-based (see, for instance, \cite{MCMC_problems2}) "Metropolis step" in a Metropolis-within-Gibbs algorithm to zero
temperature, since we are setting the randomness of the Metropolis subroutine to zero while fully retaining the randomness of the search-subspaces.
\addtocounter{algorithm}{-2}\renewcommand{\thealgorithm}{\arabic{algorithm}.1}

\begin{algorithm}[H]
\caption{Integral Geometry reweighted one-iteration ``Deterministic Nonlinear
Solver-within-Gibbs'' for sampling from Gaussians\label{alg:Great-Sphere-Deterministic}}

Steps 1-4 are the same as in Algorithm \ref{alg:Great-Sphere-Metropolis}.

\begin{enumerate}[5.]
\item for $i=1$ to $i_{\textrm{max}}$
\begin{enumerate}[(a)]
\item Generate a random isotropic $d$-dimensional linear search-subspace $S_i$ centered at the origin.  Generate a sphere $r\mathbb{S}^n$ centered at the origin with random radius $r$ distributed according to the $\chi_n$ distribution.
\item Use a deterministic nonlinear solver with random starting point to find a point $x_i$ on the intersection $S_{i+1}\cap\mathcal{M}\cap r\mathbb{S}^n$.
\item compute $\psi(x_{i})$ and the weight $w(x_{i})=\frac{f(x_i)}{|\nabla\lambda|_{r\mathbb{S}^n}(x_i)|}$.
\end{enumerate}
\end{enumerate}
\begin{enumerate}[7.]
\item Output: Weighted samples $\{x_{i}\}_{i=1}^{i_{\textrm{max}}}$
with associated weights $\{w(x_{i})\}_{i=1}^{i_{\textrm{max}}}$ that
are independent and approximately distributed according to the conditional density $f|\{\lambda(x)=c\}$, where $f(x)=\frac{1}{\sqrt{(2\pi)^{n}}}e^{-\frac{1}{2}x^{T}x}$
is the density of iid standard normals, from which we can obtain $\psi(x_{1}),\psi(x_{2}),...,\psi(x_{i_{\textrm{max}}})$
(and compute statistics of $\psi$, such as the weighted sample mean
$\frac{1}{\sum w(x_{i})}\sum w(x_{i})\times\psi(x_{i})$, the weighted
sample variance, or the weighted histogram of $\psi$)
\end{enumerate}
\end{algorithm}

\renewcommand{\thealgorithm}{\arabic{algorithm}}

\begin{rem}\label{rem:Deterministic_Approximation}
Using a deterministic solver with random starting point (Algorithm 2.1
) in place of the more random nonlinear solver-based Metropolis Markov chain subroutine of Algorithm \ref{alg:Great-Sphere-Metropolis} introduces some bias in the samples,
since the nonlinear solver probably will not find each point in the intersection $S_{i+1}\cap\mathcal{M}\cap r\mathbb{S}^n$
with equal probability.  There is nothing preventing us from using a more random Markov chain in place of the deterministic
solver, which one would normally do. However, since we only wanted
to compare weighting schemes, we can afford to use a more deterministic
solver in order to simplify numerical implementation for the time
being, as the implementation of the ``Metropolis'' step would be
beyond the scope of this paper.
It is important to note that this bias is \emph{not} a failure of
the reweighting scheme, but rather just a consequence of using a purely
deterministic solver in place of the ``Metropolis'' step. On the contrary,
we will see in Sections \ref{sec:Conditioning-on-multiple-eigenvalues} and \ref{sec:Conditioning-on-single-eigenvalue} that this bias is in fact much \emph{smaller}
than the bias present when the traditional weighting scheme
is used together with the same deterministic solver.  In the future, we plan to also perform
numerical simulations with a random Metropolis step in place of the
deterministic solver, as described in Algorithm \ref{alg:Great-Sphere-Metropolis}.
\end{rem}

\subsection{Comparison to rejection sampling \label{sub:Comparison-to-rejection-sampling}}

In this section we briefly explain why rejection sampling is oftentimes too slow when sampling from the stochastic Airy operator, implying that there is a need for a faster algorithm like Algorithm 2.1
.  Rejection sampling is slow if we condition on many eigenvalues at
a time, or if we want to condition on a large deviation of even a
single eigenvalue, since in both cases we are conditioning on rare events that occur with very low probability. In the large deviation case even the event that
the largest eigenvalue is bigger than some value (as opposed to exactly
equal to some value) is  nearly a lower-dimensional manifold, since
the eigenvalue distributions have Gaussian tails, which decay very
quickly.  For example, when we have Gaussian tails, for large $t$ the event $\{\lambda_1\geq t\}$ is approximately equal in probability to an event $\{\lambda_1\in[t,t+\epsilon)\}$ where $\epsilon\downarrow0$ as $t\rightarrow \infty$.  Hence, $\{\lambda_1\in[t,t+\epsilon)\}$ converges to the lower-dimensional manifold $\{\lambda_1=t\}$ as $t\rightarrow \infty$.

\section{Conditioning on multiple eigenvalues\label{sec:Conditioning-on-multiple-eigenvalues}}

In the first simulation (Figure \ref{fig: 6_eigenvalues}), we sampled
the fourth-largest eigenvalue conditioned on the remaining 1st- through
7th- largest eigenvalues. We begin with this example since in this
particular situation, when conditioned only on the 3rd and 5th eigenvalues,
the 4th eigenvalue is not too strongly dependent on the other eigenvalues
(the intuition for this reasoning comes from the fact that the eigenvalues
behave as a system of repelling particles with only week repulsion,
so the majority of the interaction involves the immediate neighbors
of $\lambda_{4}$). Hence, in this situation, we are able to test
the accuracy of the local solver approximation by comparison to brute
force rejection sampling. Of course, in a more general situation
where we do not have these relatively week conditional dependencies,
rejection sampling would be prohibitively slow (e.g., even if we allow
a 10\% probability interval for each of the six eigenvalues, conditioning
on all six eigenvalues gives a box that would be rejection sampled with probability $10^{-6}$).

Despite the fact that the integral geometry algorithm is solving for 6 different eigenvalues
simultaneously, the conditional probability density histogram obtained
using Algorithm 2.1 
with the integral geometry weights (Figure
\ref{fig: 6_eigenvalues}, blue) agrees closely with the conditional
probability density histogram obtained using rejection sampling (Figure
\ref{fig: 6_eigenvalues}, black). Weighting the exact same data points
obtained with Algorithm 2.1 
with the traditional weights
instead yields a probability density histogram (Figure \ref{fig: 6_eigenvalues},
red) that is much more skewed to the right than either the black or
blue curves. This is probably because, while theoretically unbiased, the traditional weights greatly amplify a small bias in the nonlinear
solver's selection of intersection points.

\begin{figure}[h]
\centering
\includegraphics[trim=0cm 3cm 0cm 0.5cm, width=11cm]{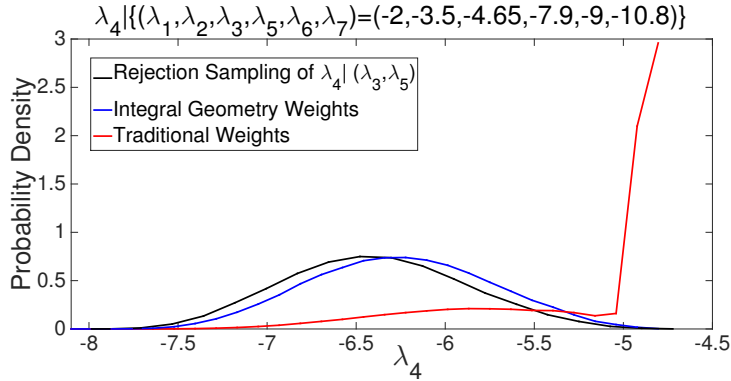}

\caption{In this simulation we used  Algorithm 2.1 
together with both the traditional weights (red) and
the integral geometry weights (blue) to plot the histogram of $\lambda_{4}|(\lambda_{1,}\lambda_{2},\lambda_{3},\lambda_{5},\lambda_{6},\lambda_{7})=(-2,-3.5,-4.65,-7.9,-9,-10.8)$
. We also provided a histogram obtained using rejection sampling of
the approximated conditioning $\lambda_{4}|(\lambda_{3,}\lambda_{5})\in[-4.65\pm0.001]\times[-7.9\pm0.001]$
(black) for comparison (conditioning on all six eigenvalues would
have caused rejection sampling to be much too slow). Since we used
a deterministic solver in place of the Metropolis subroutine in Algorithm
2, some bias is expected for both reweighting schemes. Despite this,
we see that the integral geometry histogram agrees closely with the
approximated rejection sampling histogram, but the traditional weights lead
to an extremely skewed histogram. This is probably because, while
theoretically unbiased, the traditional weights greatly amplify
a small bias in the nonlinear solver's selection of intersection points.
The skewness is especially large (in comparison to Figure \ref{fig:large_deviation})
because we are conditioning on 6 eigenvalues simultaneously.\label{fig: 6_eigenvalues}}
\end{figure}

\section{Conditioning on a single-eigenvalue rare event\label{sec:Conditioning-on-single-eigenvalue}}

In this set of simulations (Figure \ref{fig:large_deviation}), we sampled the second-largest eigenvalue
conditioned on the largest eigenvalue being equal to -2, 0, 2, and
5. Since $\lambda_{1}=5$ is a very rare event, we do not have any reasonable
chance of finding a point in the intersection of the codimension 1
constraint manifold $\mathcal{M} = \{\lambda_1=5\}$ with the search-subspace unless
we use a search-subspace of dimension $d>>1$.  Indeed, the analytical
solution for $\lambda_{1}$
tells us that $\mathbb{P}(\lambda_{1}\geq2)=1\times10^{-4}$, $\mathbb{P}(\lambda_{1}\geq4)=5\times10^{-8}$and
$\mathbb{P}(\lambda_{1}\geq5)<8\times10^{-10}$ \cite{TW_code_paper,TW_lookup_table_paper}.  For this same reason, rejection sampling for $\lambda_{1}=2$ is very slow (58
sec./sample vs. 0.25 sec./sample for Algorithm 2.1
) and we cannot hope to
perform rejection sampling for $\lambda_{1}=5$ (It would have taken about
84 days to get a single sample!). To allow us to make a histogram
in a reasonable amount of time, we will use Algorithm 2.1 
with search-subspaces of dimension
$d=23>>1$, vastly increasing the probability of the random search
subspace intersecting $\mathcal{M}$.

In (Figure \ref{fig:large_deviation}, top), we
see that while the rejection sampling (black) and integral geometry weight
(blue) histograms of the density of $\lambda_{2}|\lambda_{1}=2$
are fairly close to each other, the plot obtained with the exact same
data as the blue plot but weighted in the traditional way (red)
is much more skewed to the right and less smooth than both the black
and blue curves, implying that using the integral geometry weights from Theorem \ref{th:1} greatly
reduces bias and increases the convergence speed (Although the red curve is not as
skewed as in Figure \ref{fig:large_deviation} of Section \ref{sec:Conditioning-on-multiple-eigenvalues}.
This is probably because in this situation the codimension of $\mathcal{M}$
is 1, while in Section \ref{sec:Conditioning-on-multiple-eigenvalues}
the codimension was $6$.)

In (Figure \ref{fig:large_deviation}, middle), where we conditioned instead on $\lambda_1=5$, we
see that solving from a random starting point but not restricting
oneself to a random search-subspace (purple plot) causes huge errors
in the histogram of $\lambda_{2}|\lambda_{1}$. We also see that,
as in the case of $\lambda_{1}=2$, the plot of $\lambda_{2}$ obtained
with the traditional weights is much more skewed to the right
and less smooth than the plot obtained using the integral geometry
weights.

In (Figure \ref{fig:large_deviation}, bottom), we use our Algorithm 2.1 
to study the behavior of $\lambda_2|\lambda_1$ for values of $\lambda_1$ at which it would be difficult to obtain accurate curves with traditional weights or rejection sampling.  We see that as we
move $\lambda_{1}$ to the right, the variance of $\lambda_{2}|\lambda_{1}$
increases and the mean shifts to the right. One explanation for this
is that the largest and third-largest eigenvalues normally repel the
second-largest eigenvalue, squeezing it between the largest- and third-
largest eigenvalues, which reduces the variance of $\lambda_{2}|\lambda_{1}$.
Hence, moving the largest eigenvalue to the right effectively
"decompresses" the probability density of the second-largest eigenvalue, increasing it's variance.  Moving the largest eigenvalue to the right also allows the second-largest eigenvalue's mean to move to the right by reducing the repulsion from the right caused by the largest eigenvalue.




\begin{figure}
\centering
\includegraphics[trim=0cm 0.3cm 0cm 2cm, clip=true,width=13cm]{large_deviations_6a_squished-eps-converted-to.pdf}
\includegraphics[trim=0cm 5cm 0cm 6cm, clip=true,width=13cm]{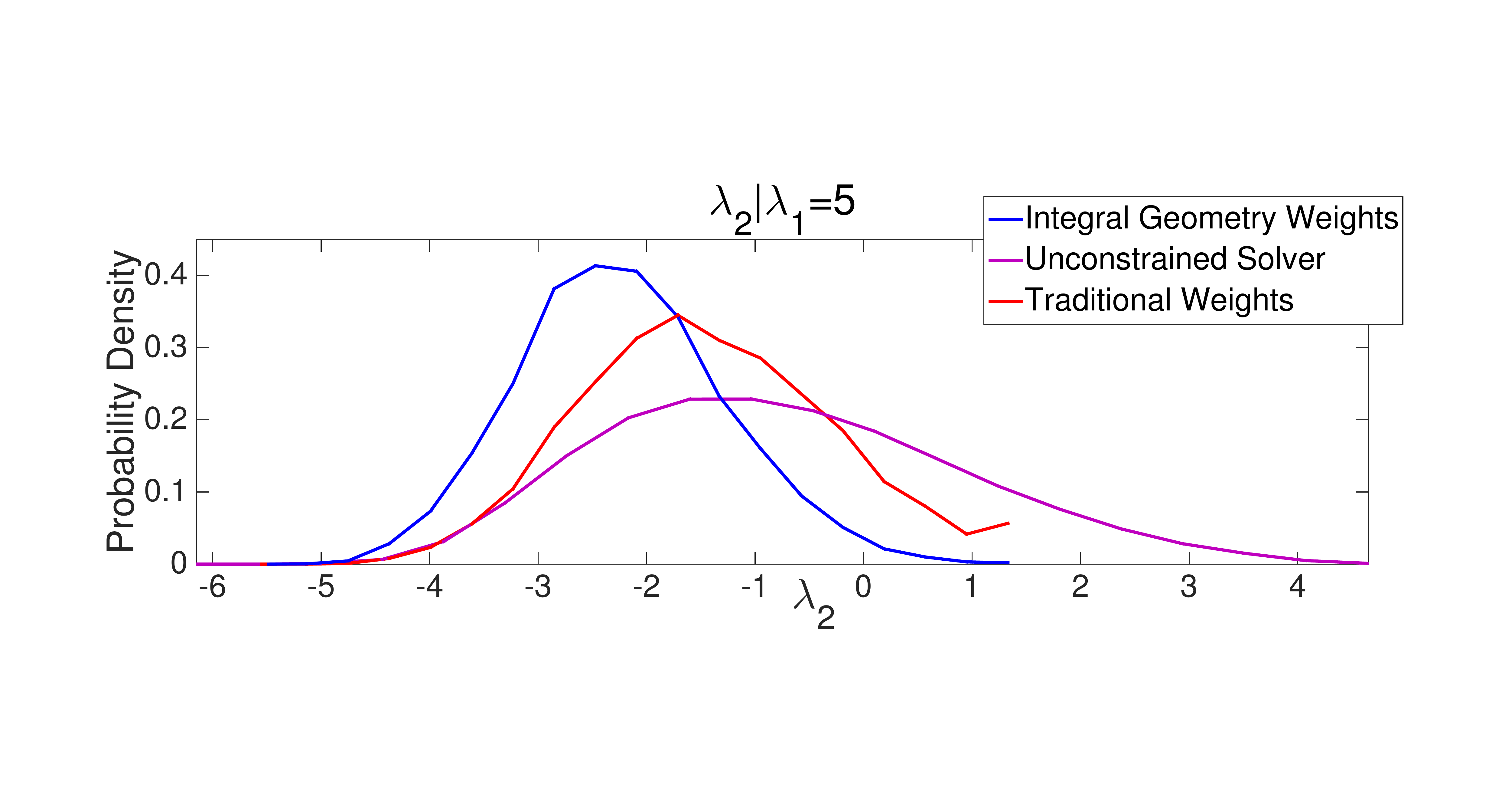}
\includegraphics[trim=0cm 5cm 0cm 5cm, clip=true,width=13cm]{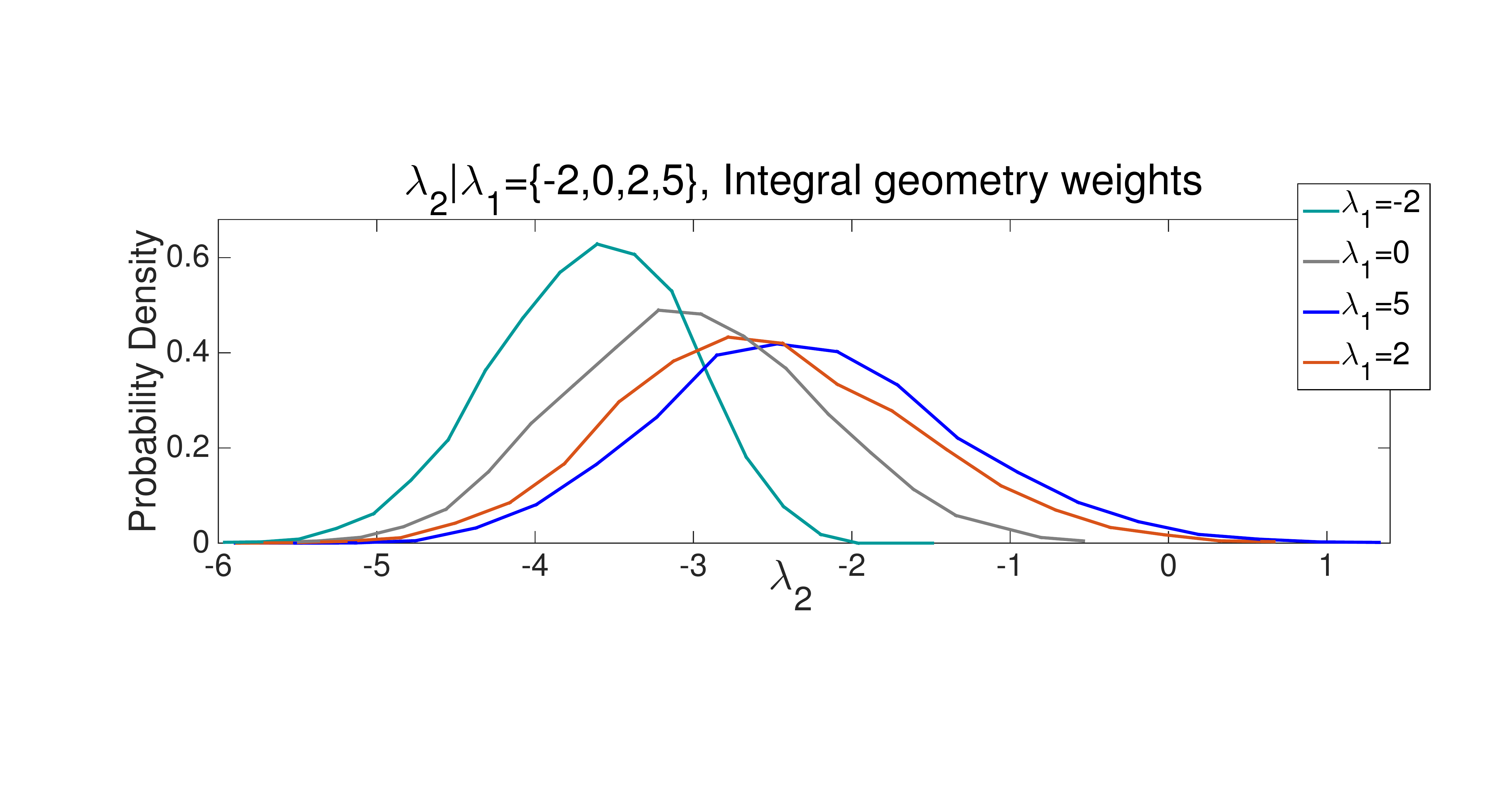}

\caption{Histograms of $\lambda_{2}|\lambda_{1}=-2, 0, 2,\textrm{ and }5$, generated using Algorithm 2.1
. A search-subspace of dimension $d=23$ was used, allowing
us to sample the rare event $\lambda_{1}=5$. In the first plot (top) we see that the rejection sampling histogram
of $\lambda_{2}|\lambda_{1}=2$ is much closer to the histogram obtained
using the integral geometry weights (blue) than the histogram obtained
with the traditional weights (red) because the red plot is much
more skewed to the right and less smooth (it takes longer to converge)
than either the blue or black plots.  If we do not constrain the solver to a random search
subspace, the histogram we get for $\lambda_{2}|\lambda_{1}=5$ (purple)
is very skewed to the right (middle plot),
implying that using random search-subspace (as a opposed to just a
random starting point) greatly helps in the mixing of our samples
towards the correct distribution. In the last plot (bottom),
the probability densities of $\lambda_{2}|\lambda_{1}$ obtained with
the integral geometry weights show that moving the largest eigenvalue
to the right has the effect of increasing the variance of the probability
density of $\lambda_{2}|\lambda_{1}$ and moving its mean to the right,
probably because the second eigenvalue feels less repulsion from the
largest eigenvalue as $\lambda_{1}\rightarrow\infty$.\label{fig:large_deviation}}
\end{figure}

\begin{rem}
As discussed in Remark \ref{rem:Deterministic_Approximation} of Section \ref{sub:Approximate-sampling-algorithm-implementation}, if we wanted to get a perfectly
accurate plot, we would still need to use a randomized solver, such as a Metropolis-Hastings solver, to randomize over the intersection points.
Since $d=23$, the volumes of the exponentially many connected submanifolds in the intersection
$S_{i+1}\cap\mathcal{M}$ would be concentrated in just a few of these
submanifolds, with the concentration being exponential in $d$, causing
the algorithm to be prohibitively slow for $d=23$ unless we use Algorithm \ref{alg:Curvature}, which uses the Chern-Gauss-Bonnet curvature reweighting of Theorem \ref{th:2} (see Section \ref{sub:Sphere-example}).
Hence, if we were to implement the randomized solver of Algorithm \ref{alg:Great-Sphere-Metropolis},
the red curve would converge extremely slowly unless we reweighted
according to Theorem \ref{th:2} (in addition to Theorem \ref{th:1}). Hence, the situation for the traditional weights
is in fact much worse in comparison to the integral geometry weights
of Theorems \ref{th:1} and \ref{th:2} than even (Figure \ref{fig:large_deviation}, middle) would suggest.
\end{rem}

\section{Acknowledgments:}

The research of Oren Mangoubi is supported by the Department of Defense (DoD) through the National Defense Science \& Engineering Graduate Fellowship (NDSEG) Program and by the MIT Mathematics department. The research of Alan Edelman is supported by NSF DMS-1312831.

We are extremely grateful to Michael Lacroix, Jiahao Chen, Natesh Pillai, Aaron Smith, and Jonathan Kelner for insightful discussions and advice.

\FloatBarrier

\bibliographystyle{plain}
\addcontentsline{toc}{section}{\refname}
\bibliography{oren}

\begin{thebibliography}{10}

\bibitem{MCMC_Application_molecular_biology}
{Anders S. Christensen, Troels E. Linnet, Mikael Borg, Wouter Boomsma, Kresten
  Lindorff-Larsen, Thomas Hamelryck and Jan H. Jense}.
\newblock Protein structure validation and refinement using amide proton
  chemical shifts derived from quantum mechanics.
\newblock {\em PLoS ONE}, 8(12):1--10, 2013.

\bibitem{Singer}
Michael~F. Atiyah and Isadore Singer.
\newblock The index of elliptic operators on compact manifolds.
\newblock {\em Bulletin of the American Mathematical Society}, 69(3):422--433,
  1963.

\bibitem{MCMC_review}
Julian Besag.
\newblock {M}arkov chain {M}onte {C}arlo for statistical inference.
\newblock Technical report, University of Washington, Department of Statistics,
  04 2001.

\bibitem{TW_code_paper}
Folkmar Bornemann.
\newblock On the numerical evaluation of distributions in random matrix theory:
  A review.
\newblock {\em Markov Processes and Related Fields}, 16(4):803--866, 2010.

\bibitem{MCMC_Application_Machine_Learning}
A.~Doucet C.~Andrieu, N. de~Freitas and M.I. Jordan.
\newblock An introduction to {MCMC} for machine learning.
\newblock {\em Machine Learning}, 50:5--43, 2003.

\bibitem{Chern}
Shiing-Shen Chern.
\newblock On the curvatura integra in a {R}iemannian manifold.
\newblock {\em Annals of Mathematics}, 46(4):674--684, 1945.

\bibitem{MCMC_Application_General_Relativity}
Neil~J. Cornish and Edward~K. Porter.
\newblock {MCMC} exploration of supermassive black hole binary inspirals.
\newblock {\em Classical Quantum Gravity}, 23(19):761--767, 2006.

\bibitem{Crofton}
Morgan~W. Crofton.
\newblock On the theory of local probability, applied to straight lines drawn
  at random in a plane; the methods used being also extended to the proof of
  certain new theorems in the integral calculus.
\newblock {\em Philosophical Transactions of the Royal Society of London},
  158:181--199, 1868.

\bibitem{Stochastic_Operators}
Alan Edelman and Brian~D. Sutton.
\newblock From random matrices to stochastic operators.
\newblock {\em Journal of Statistical Physics}, 127(6):1121--1165, 2007.

\bibitem{Gelfand}
Israel~M. Gelfand and Mikhail~M. Smirnov.
\newblock Lagrangians satisfying {C}rofton formulas, {R}adon transforms, and
  nonlocal differentials.
\newblock {\em Advances in Mathematics}, 109(2):188--227, 1994.

\bibitem{MCMC_MLE}
Charles~J. Geyer.
\newblock Markov chain {M}onte {C}arlo maximum likelihood.
\newblock In {\em Computing Science and Statistics: Proceedings of the 23rd
  Symposium on the Interface}, pages 156--163, 1991.

\bibitem{Guth}
Larry Guth.
\newblock Degree reduction and graininess for {K}akeya-type sets in
  {$\mathbb{R}^3$}.
\newblock {\em preprint on arXiv:1402.0518}, 2014.

\bibitem{Helgason}
Sigurdur Helgason.
\newblock {\em Integral Geometry and Radon Transforms}.
\newblock Springer, New York, 2010.

\bibitem{MCMC_Application_genetics}
Jody Hey and Rasmus Neilsen.
\newblock Integration within the {F}elsenstein equation for improved {M}arkov
  chain {M}onte {C}arlo methods in population genetics.
\newblock {\em Proceedings of the national academy of sciences of the United
  States of America}, 104(8):2785--2790, 2006.

\bibitem{Optimization_in_MCMC}
{J. Martin, L. Wilcox, C. Burstedde and O. Ghattas}.
\newblock A stochastic {N}ewton {MCMC} method for large-scale statistical
  inverse problems with application to seismic inversion.
\newblock {\em SIAM Journal on Scientific Computing}, 34(4):A1460--A1487, 2012.

\bibitem{Paiva}
E.~Fernandes J.C. Alvarez~Paiva.
\newblock Gelfand transforms and {C}rofton formulas.
\newblock {\em Selecta Mathematica}, 13(3):369--390, 2008.

\bibitem{Johnstone}
I.M. Johnstone.
\newblock On the distribution of the largest eigenvalue in principal components
  analysis.
\newblock {\em Annals of Statistics}, 29:295--327, 2001.

\bibitem{RRV}
{Jos{\'e} Ram{\'\i}rez, Brian Rider and B{\'a}lint Vir{\'a}g}.
\newblock Beta ensembles, stochastic {A}iry spectrum, and a diffusion.
\newblock {\em Journal of the American Mathematical Society}, 24(4):919--944,
  2011.

\bibitem{Graphical_Models_Book}
Daphne Koller and Nir Friedman.
\newblock {\em Probabilistic Graphical Models}.
\newblock MIT Press, Cambridge, 2009.

\bibitem{Concentration_of_measure_phenomenon}
Michel Ledoux.
\newblock The concentration of measure phenomenon.
\newblock In {Peter Landweber, Michael Loss, Tudor Ratiu and J.T. Stafford},
  editor, {\em Mathematical Surveys and Monographs}, volume~89. American
  Mathematical Society, 2001.

\bibitem{Tubular_Neighborhoods}
Martin Lotz.
\newblock On the volume of tubular neighborhoods of real algebraic varieties.
\newblock {\em arXiv:1210.3742}, 2012.

\bibitem{Oren_Concentration_of_Measure}
Oren Mangoubi and Alan Edelman.
\newblock Concentration of kinematic measure.
\newblock {\em In Preparation}, 2015.

\bibitem{Milman}
Vitali~D. Milman.
\newblock A new proof of {A.} {D}voretzky's theorem on cross-sections of convex
  bodies.
\newblock {\em Funkcional. Anal. i Prilozhen}, 5(4):28--37, 1971.

\bibitem{BayesianBook}
{Ming-Hui Chen, Qi-Man Shao and Joseph G. Ibrahim}.
\newblock {\em Monte Carlo methods in Bayesian computation}.
\newblock Springer, 2000.

\bibitem{TW_lookup_table_paper}
Boaz Nadler.
\newblock On the distribution of the ratio of the largest eigenvalue to the
  trace of a {W}ishart matrix.
\newblock {\em Journal of Multivariate Analysis}, 102, 2011.

\bibitem{Handbook}
Radford~M. Neal.
\newblock {MCMC} using {H}amiltonian dynamics.
\newblock In {Steve Brooks, Andrew Gelman, Galin L. Jones and Xiao-Li Meng},
  editor, {\em Handbook of {M}arkov Chain {M}onte {C}arlo}. CRC Press, 2011.

\bibitem{Total_Absolute_Curvature_2}
Stepan~Yu. Orevkov.
\newblock Sharpness of {R}isler's upper bound for the total curvature of an
  affine real algebraic hypersurface.
\newblock {\em Russian Mathematical Surveys}, 62:393--394, 2007.

\bibitem{Manifold_MCMC_Diaconis}
{Persi Diaconis, Susan Holmes and Mehrdad Shahshahani}.
\newblock Sampling from a manifold.
\newblock In {\em Advances in Modern Statistical Theory and Applications: A
  Festschrift in Honor of Morris L. Eaton}, pages 102--125. Institute of
  Mathematical Statistics, 2013.

\bibitem{MCMC_Application_Linguistics}
{R.H. Baayen, D.J. Davidson and D.M. Bates}.
\newblock Mixed-effects modeling with crossed random effects for subjects and
  items.
\newblock {\em Journal of Memory and Language}, 59:390--412, 2008.

\bibitem{Total_Absolute_Curvature_1}
Jean-Jaques Risler.
\newblock On the curvature of the real milnor fiber.
\newblock {\em Bull. London Math. Soc.}, 35(4):445--454, 2003.

\bibitem{Santalo}
Luis~A. Santalo.
\newblock Integral geometry and geometric probability.
\newblock In Gian-Carlo Rota, editor, {\em Encyclopedia of Mathematics and its
  Applications}, volume~1. Addison-Wesley Publishing Company, 1976.

\bibitem{Stochastic_Geometry_Book}
Rolf Schneider and Wolfgang Weil.
\newblock {\em Stochastic and Integral Geometry}.
\newblock Springer, 2008.

\bibitem{SpivakIII}
Michael Spivak.
\newblock {\em A comprehendsive introduction to differential geometry, Vol.
  III}.
\newblock Publish or Perish, Inc., Berkeley, 1999.

\bibitem{SpivakV}
Michael Spivak.
\newblock {\em A comprehendsive introduction to differential geometry, Vol. V}.
\newblock Publish or Perish, Inc., Berkeley, 1999.

\bibitem{Sutton_Thesis}
Brian~D. Sutton.
\newblock {\em {The stochastic operator approach to random matrix theory}}.
\newblock PhD thesis, Massachusetts Institute of Technology, 2011.

\bibitem{Gauss_Bonnet_for_Hypersurfaces}
Mihai Tibar and Dirk Siersma.
\newblock Curvature and {G}auss-{B}onnet defect of global affine hypersurfaces.
\newblock {\em Bulletin des Sciences Mathematiques}, 130(2):110--122, 2006.

\bibitem{MCMC_manifold_application_statistical_mechanics}
{Tony Lelièvre, Mathias Rousset, and Gabriel Stoltz}.
\newblock {\em Free energy computations: A Mathematical Perspective}.
\newblock Imperial College Press, 2010.

\bibitem{Trefethen_book}
L.~N. Trefethen and D.~Bau III.
\newblock {\em Numerical Linear Algebra}.
\newblock SIAM, 1997.

\bibitem{MCMC_problems2}
{Yongtao Guan and Stephen M. Krone}.
\newblock Small-world {MCMC} and convergence to multi-modal distributions: from
  slow mixing to fast mixing.
\newblock {\em The Annals of Applied Probability}, 17(1):284--304, 2007.

\bibitem{Gauss_Bonnet_for_Hypersurfaces2}
Chenchang Zhu.
\newblock The {G}auss-{B}onnet theorem and its applications.
\newblock {\em http://math.berkeley.edu/~alanw/240papers00/zhu.pdf}, 2004.

\end{thebibliography}

\end{document}